\documentclass[a4paper,12pt]{amsart}
\usepackage{amsmath,amssymb,amsthm,epsfig,epstopdf,url,array,hyperref}

\title{ 
Counting polynomials with positive roots}
\author{Pavlo Yatsyna}
\address{Charles University, Faculty of Mathematics and Physics, Department of Algebra, Sokolov\-sk\' a 83, 18600 Praha~8, Czech Republic}
\email{p.yatsyna@matfyz.cuni.cz}
\author{B\l{}a\.zej \.Zmija}
\address{Charles University, Faculty of Mathematics and Physics, Department of Mathematical Analysis, Sokolov\-sk\' a 83, 18600 Praha~8, Czech Republic \newline
and Institute of Mathematics of the Polish Academy of Sciences, \'{S}niadeckich 8, 00-656 Warsaw, Poland}
\email{blazej.zmija@gmail.com}

\thanks{The first author was supported by {Charles University} programmes PRIMUS/24/SCI/010 and UNCE/24/SCI/022. The second author was supported by {Charles University} PRIMUS Research Programme PRIMUS/25/SCI/017.}

\newtheorem{theorem}{Theorem}
\newtheorem{lemma}[theorem]{Lemma}

\newtheorem{cor}[theorem]{Corollary}
\newtheorem{prop}[theorem]{Proposition}

\newtheorem{remark}[theorem]{Remark}

\newcommand{\R}{\mathbb{R}}
\newcommand{\N}{\mathbb{N}}
\newcommand{\Z}{\mathbb{Z}}
\newcommand{\dt}{{\rm d}t}
\newcommand{\rad}{\mathrm{rad}\hspace{0.05cm}}

\allowdisplaybreaks

\begin{document}

\begin{abstract}
    This paper investigates the number of monic integer polynomials of degree $n$ whose roots are all real and positive. We establish an asymptotic formula for the case of fixed trace by estimating the number of integer sequences satisfying Maclaurin's inequalities. For cubic polynomials, we derive a much more precise asymptotic result. Furthermore, we analyse the arithmetic properties of the discriminants of these polynomials, showing that a positive proportion of cubics have square-free discriminants.
\end{abstract}
\maketitle
\section{Introduction}

Cohn was one of the first to study the number of cubic number fields \cite{cohn}, showing that abelian fields are rare. This was followed by the seminal works of Davenport and Heilbronn \cite{DH1,DH2}. In more recent years, significant progress has been made in counting small-degree number fields, particularly in works such as \cite{bhargava05} by Bhargava on quartic and quintic fields. Research on cubic fields continues to develop,  as seen, for example, in \cite{CD}, \cite{Xiao}, which addresses cubic polynomials with a given Galois group, or abelian fields of trace one \cite{BO}. Notably, it remains an open question whether infinitely many cubic fields with prime discriminants exist. 

For a generic polynomial, the expected number of real roots of an equation
\[
A_0X^n-A_1X^{n-1}+\cdots+(-1)^nA_n=0,
\] 
where each $A_i\in\{-1,1\}$ chosen independently with equal probability, is at most $25(\log n)^2+12 \log n$, for $n$ sufficiently large \cite{LO}. If instead the coefficients $A_i$ are normally distributed with the density function $e^{-u^2/\pi^{1/2}}$, following Kac's formula \cite{Kac} the expected number of real roots is approximately $(2/\pi)\log n$. Several recent works, including those by Dubickas and Sha \cite{DS}, Bertók, Hajdu, and Pethő \cite{BHP}, and Dubickas \cite{Dubickas18}, have studied the asymptotic number of integer polynomials with coefficients bounded in absolute value by $H$, as $H \to \infty$, under various constraints. In particular, Dubickas \cite{Dubickas18} showed that the number of such polynomials with a fixed number of complex roots is different depending on whether polynomials have real roots or not. For irreducible polynomials (see, for example, \cite{Dubickas14}, for the reducible case), the restriction to integer coefficients corresponds to studying algebraic numbers that are roots of those polynomials. Polynomials with fixed leading and trailing coefficients were counted by Grizzard and Gunther in \cite{GG}, with respect to the multiplicative Weil height, without restrictions on the number of real roots.

For $n \in \mathbb{N}$, let $\mathcal{P}_n$ denote the set of all monic integer polynomials of degree $n$ with all roots real. For a given $k\in \N$, with $k<n$, and $A_1,\ldots,A_k\in \Z$, let $\mathcal{P}_n(A_1,\ldots,A_k)$ denote the subset of $\mathcal{P}_n$ consisting of polynomials whose first $k+1$ coefficients are fixed to $1,-A_1,\ldots,(-1)^kA_k$, as
\[
\mathcal{P}_n(A_1,\ldots,A_k):=\left\{X^n-A_1X^{n-1}+\cdots +(-1)^k A_{k}X^{n-k}\cdots=\prod^n_{i=1}(X-\alpha_i) \in\mathbb{Z}[X] \Bigg|  \alpha_i\in\mathbb{R}\right\}.
\]
Let $\mathcal{P}_n^+(A_1,\ldots,A_k)\subset \mathcal{P}_n(A_1,\ldots,A_k)$ be the subset of all polynomials with all roots real and positive. In this paper, we will be interested in counting the number of polynomials in $\mathcal{P}_n^+(A_1,\ldots,A_k)$. The coefficients of such polynomials satisfy \textit{Maclaurin's inequalities} for the elementary symmetric means, i.e. for $e_k=e_k(\alpha_1,\ldots, \alpha_n)=\sum_{1\le i_1<\cdots<i_k\le n} \alpha_{i_1}\cdots\alpha_{i_k}$, where $\alpha_j \in \R^+$, it holds that

\[\dfrac{e_1}{\binom{n}{1}}\ge \sqrt{\dfrac{e_2}{\binom{n}{2}}}\ge\ldots \ge \sqrt[k]{\dfrac{e_k}{\binom{n}{k}}}\ge \ldots \ge \sqrt[n]{\dfrac{e_n}{\binom{n}{n}}}.\] 

Hence, the sequence $(A_1,\ldots,A_k,\ldots)$ is \textit{attainable} (with respect to Maclaurin's inequality) if

\[\dfrac{A_1}{\binom{n}{1}}\ge \sqrt{\dfrac{A_2}{\binom{n}{2}}}\ge\ldots \ge \sqrt[k]{\dfrac{A_k}{\binom{n}{k}}}\ge \ldots \] 
holds.
The number of attainable sequences is at least as large as the number of monic integer polynomials with only positive real roots, and thus bounds from above the cardinality of $\mathcal{P}_n^+(A_1,\ldots,A_k)$. Here, the attainable sequences are the $n$-tuples $(B_1,\ldots,B_n)$ of natural numbers that satisfy Maclaurin's inequality, where $B_i=A_i$ for $i\le k$. This bound is sharp for quadratic polynomials, but even for $n=3$, there exist attainable sequences for which the polynomial does not have all real roots (for example, $X^3-4X^2+3X-1$). Better estimates can be obtained using improved inequalities similar to Maclaurin's; see, in particular, \cite{Tao}. The paper's first result is the following:

\begin{theorem}\label{ThmIneq1}
For every $n\geq 2$ and $k$ define $S_{n,k}(x) := \frac{x}{\binom{n}{k}}$ and $s_{n,k}(x):=S_{n,k}(x)^{1/k}$. Let $A=A_1 \in \N$ and
\begin{align*}
\mathcal{S}_{n}(A):= \left\{ (A_{2}, \ldots ,A_{n})\in \mathbb{N}^{n-1} | s_{n,1}(A_{1})\geq s_{n,2}(A_{2})\geq s_{n,3}(A_{3})\geq\ldots \geq s_{n,n}(A_{n}) \right\}.
\end{align*}
Then
\begin{align*}
\#\mathcal{S}_{n}(A) =  \left(\Phi_{n} + O_{\leq 1}\left(\Psi_{n} A^{-2}\right) \right) A^{\frac{(n-1)(n+2)}{2}},
\end{align*}
where ``$\leq 1$" in the subscript means that the implied constant is at most $1$, and
\begin{align*}
\Phi_{n} & := \frac{2^{n-1}}{n^{\frac{1}{2}(n^{2}+n-4)}}\frac{(n+1)!}{(2n)!} \prod_{k=1}^{n-2}\binom{n}{k}, \\
\Psi_{n} & := (n-1)\frac{n^{\frac{n}{n-1}}}{n^{\frac{(n-1)(n+2)}{2}}}\cdot\prod_{k=1}^{n-2}\binom{n}{k}^{1+\frac{1}{k}}.
\end{align*}
In particular,
\begin{align*}
\#\mathcal{S}_{n}(A)=\left(\frac{eA}{n}\right)^{\left(\frac{1}{2}+o_{n\to\infty}(1)\right)n^{2}},
\end{align*}
where $o_{n\to\infty}(1)$ is a quantity that tends to $0$ when $n\to\infty$ and does not depend on $A$.
\end{theorem}

And, as expected, here is the corollary for the count of monic polynomials with positive real roots:

\begin{cor}\label{CorIneq1}
    Let $A\in \N$, then
    \begin{align*}
       \#\mathcal{P}_{n}^{+}(A) \leq \left(\Phi_{n} +\Psi_{n} A^{-2} \right) A^{\frac{(n-1)(n+2)}{2}},
    \end{align*}
    where $\Phi_{n}$ and $\Psi_{n}$ are the same as in Theorem \ref{ThmIneq1}. In particular,
    \begin{align*}
       \#\mathcal{P}_{n}^{+}(A) \leq \left(\frac{eA}{n}\right)^{\left(\frac{1}{2}+o_{n\to\infty}(1)\right)n^{2}},
    \end{align*}
    where $o_{n\to\infty}(1)$ is a quantity that tends to $0$ when $n\to\infty$ and does not depend on $A$.
\end{cor} 

The restriction to positive roots above is necessary, as $\mathcal{P}_n(A)$ is infinite for $n\ge 2$ and any $A\in \Z$. We have the following, somewhat surprising result, which is derived from the application of the above-mentioned improved version of the Maclaurin type inequality by \cite{Tao}: 

\begin{theorem}\label{ThmGeneralRoot}
    Let $A_1,A_2, n \in \N$, and
    \begin{align*}
        M(A_{1},A_{2}):= \max\left\{\frac{1}{n}|A_{1}|,\left(\frac{1}{n(n-1)}|A_{2}|\right)^{1/2}\right\}.
    \end{align*}
    Then
    \begin{align*}
       \#\mathcal{P}_{n}(A_{1},A_{2}) \leq \prod_{k=1}^{n-3} \binom{n}{k} \cdot \left(\prod_{k=3}^{n} k^{k}\right)^{\frac{1}{2}} \cdot (3C)^{\frac{n^{2}+n-6}{2}}\cdot M(A_{1},A_{2})^{\frac{n^{2}+n-6}{2}},
    \end{align*}
    where $C=160e^{7}\approx 175461$.
    
    In particular,
    \begin{align*}
       \#\mathcal{P}_{n}(A_{1},A_{2})\leq \left(\frac{3Cn M(A_{1},A_{2})}{e^{1/2}}\right)^{\left(1+o_{n\to\infty}(1)\right)\frac{n^{2}}{2}} \leq M_{1}(A_{1},A_{2}) ^{\left(1+o_{n\to\infty}(1)\right)\frac{n^{2}}{2}},
    \end{align*}
    where $M_{1}(A_{1},A_{2}) := 3C\max\left\{|A_{1}|,|A_{2}|^{1/2}\right\}$ does not depend on $n$.
\end{theorem}

More can be said about the number of polynomials with real and positive roots for low degrees. For degree one, there is only a single monic linear integer polynomial with a given trace. In the quadratic case, this is equivalent to counting the admissible tuples, as given by Maclaurin’s inequality, i.e. $\mathcal{P}_2^+(A)=\lceil\frac{1}{4}A^2\rceil$. However, for cubic polynomials, the problem becomes significantly more complex and interesting, as noted earlier. Our upper bound from Corollary \ref{CorIneq1} for $n=3$ gives $\mathcal{P}_{3}^{+}(A)\leq \frac{2}{405}A^{5}+\frac{2}{3^{3/2}}A^{3}$. In fact, we can obtain a much more precise result by carefully examining Robinson's algorithm (as seen here \cite{MS}):

\begin{theorem}\label{ThmCubic}
    Let $A\in \N$, then
   \begin{align*}
       \#\mathcal{P}_{3}^{+}(A) &= \frac{1}{480}A^{5} + O_{\leq 1}\left(2A^{3} \right).
    \end{align*}
\end{theorem}
Even stronger results of a similar nature are presented in Theorem \ref{ThemCubicabc} and Corollary \ref{Cor3Coeffs}. For cubic polynomials, we can also study other arithmetic properties, such as how often the discriminant is square-free. A positive proportion (which tends to $30.7056\%$ as the degree goes to infinity) of monic integer polynomials has a square-free discriminant \cite{BHP}. The upper bounds for the number of monic polynomials of bounded height and a given discriminant of the corresponding simple extension are due to \cite{DOS}.

We show that a positive proportion of cubic polynomials with positive roots has a square-free discriminant. We denote the discriminant of a polynomial $f$ by $\Delta_f$, defined as the resultant $\mathrm{Res}(f,f')$, where $f'$ is the derivative of $f$:

\begin{theorem}\label{ThmCubicSquarefreeDelta}
    Let $A\in \N$, and
    \begin{align*}
       \mathcal{P}_{3}^{\square+}(A) &:=\big\{f \in\mathcal{P}^+_3(A)\mid \Delta_f \text{ is square-free}\}.
    \end{align*}
    Then
    \begin{align*}
       \#\mathcal{P}_{3}^{\square+}(A) & \geq 3\cdot 10^{-5} A^{5} + O\left(\frac{A^{5}}{(\log\log A)^{2}}\right). 
    \end{align*}
\end{theorem}

In the case of the set $\mathcal{P}_3(A,B)$ with restricted discriminant, we have the following result inspired by the recent work of Badziahin \cite{Badziahin}:

\begin{theorem}\label{ThmUpBoundW1}
    Let $A, B, D\in \N$, and
    \begin{align*} 
        \mathcal{P}_3(A,B;D):=\big\{ f\in \mathcal{P}_3(A,B)\mid \Delta_{f}\leq D \big\}.
    \end{align*}
    Then
    \begin{align*}
        \#\mathcal{P}_{3}(A,B;D) \leq \left\{\begin{array}{ll}
            \frac{4}{27}\left(A^{2}-3B\right)^{3/2}, & \textrm{if } D\geq \frac{4}{27}\left(A^{2}-3B\right)^{3}, \\
            \frac{D}{\left(A^{2}-3B\right)^{3/2}}, & \textrm{if } D< \frac{4}{27}\left(A^{2}-3B\right)^{3}.
        \end{array}\right.
    \end{align*}
    In particular,
    \begin{align*}
       \# \mathcal{P}_3(A,B;D)\leq \frac{2}{3\sqrt{3}}\sqrt{D}.
    \end{align*}
\end{theorem}

Another condition we may impose on $\Delta_f$ is that it has few prime factors. Applying a result of Iwaniec \cite{Iwaniec}, we derive the following consequence of Theorem \ref{ThmAlmostPrimeDisc} and Proposition \ref{PropAlmostPrime}:

\begin{theorem}\label{ThmAlmostPrimeDiscLowBound}
    For every $H>1$ there are 
    \begin{align*}
        \gg \frac{H^{3}}{(\log H)^{2}}
    \end{align*}
    integral triples $(A,B,C)\in [-H,H]^{3}$ such that the discriminant of the polynomial $X^{3}-AX^{2}+BX-C$ has at most two prime factors.
\end{theorem}

\section*{Notation}

Throughout the paper, we will use the standard big $O$ and small $o$ notations: given functions $f(x)$ and $g(x)$ we write $f(x)=O(g(x))$ if there exists a constant $C>0$ such that $f(x)<Cg(x)$ for all $x$, and $f(x)=o(g(x))$ if $\lim_{x\to\infty}\frac{f(x)}{g(x)} = 0$. We will also utilise the Vinogradov notation: $f(x)\ll g(x)$ if and only if $f(x)=O(g(x))$. The symbol $f(x)\sim g(x)$ means $\lim_{x\to\infty}\frac{f(x)}{g(x)}=1$.

Sometimes we will need more precise information: $f(x)=O_{\leq 1}(g(x))$ means that $f(x)\leq g(x)$ (so the implied constant is less than or equal to $1$). If functions in more than one variable are considered, the symbol $o_{n\to\infty}(1)$ denotes a quantity that goes to $0$ as $n\to\infty$.

\section*{Acknowledgments}

We are extremely grateful to Vítězslav Kala for his many valuable suggestions, and to Giacomo Cherubini for his insightful discussions.

\section{General upper bound}

In this section, we present (in Theorem \ref{ThmIneqMain} below) a generalisation of Theorem \ref{ThmIneq1}. We provide this statement because we believe that it may be helpful in problems where inequalities similar to Maclaurin's inequality occur. It is possible to move even further and replace the exponents $\frac{k+1}{k}$ by terms of a fixed sequence. However, the formulas become more complicated, so we skip it.

The specialisation of Theorem \ref{ThmIneqMain} to Theorem \ref{ThmIneq1} is shown in the next section.

\begin{theorem}\label{ThmIneqMain}
Consider a sequence $\mathcal{B}=(B_{1},\ldots , B_{n})$ of positive numbers and length $n\geq 2$. Let $A=A_1\in \mathbb{N}$, and
\begin{align*}
\mathcal{S}_{\mathcal{B}}(A):= \left\{(A_{2},\ldots ,A_{n})\in \mathbb{N}^{n-1} \bigg| A_{k+1}\leq B_{k}A_{k}^{\frac{k+1}{k}} \textrm{ for every } k\in\{1,\ldots ,n-1\} \right\}.
\end{align*}    
Then:
\begin{align*}
\#\mathcal{S}_{\mathcal{B}}(A)=\left(\Phi_{\mathcal{B}}+O_{\leq 1}(\Psi_{\mathcal{B}}A^{-2})\right)A^{\frac{(n-1)(n+2)}{2}},
\end{align*}
where
\begin{align*}
\Phi_{\mathcal{B}} & = 2^{n-1}n\frac{(n+1)!}{(2n)!}\prod_{k=1}^{n-1}B_{k}^{\frac{(n-k)(n+k+1)}{2(k+1)}}, \\
\Psi_{\mathcal{B}} & = (n-1)\prod_{k=1}^{n-2}B_{k}^{\frac{(n-k-1)(n+k+2)}{2(k+1)}}\cdot\prod_{k=1}^{n-2}\max\left\{1,B_{k}^{-\frac{1}{k+1}}\right\}.
\end{align*}
\end{theorem}

We will need some lemmas:

\begin{lemma}[Euler's summation formula]\label{LemEuler}
    If $f$ has continuous derivative $f'$ on the interval $[y,x]$, where $0<y<x$, then
    \begin{align*}
        \sum_{y<n\leq x} f(n) = \int_{y}^{x}f(t)\dt + \int_{y}^{x}(t-\lfloor t\rfloor )f'(t)\dt -(x-\lfloor x\rfloor)f(x) + (y-\lfloor y\rfloor )f(y).
    \end{align*}
\end{lemma}
\begin{proof}
    See \cite[Theorem 3.1]{Apostol}.
\end{proof}

\begin{lemma}
For every $\alpha >0$ and $x\geq 1$: 
\begin{align*}
\sum_{n\leq x}n^{\alpha} = \frac{x^{\alpha+1}}{\alpha +1} + O_{\leq 1}(x^{\alpha})
\end{align*}
and
\begin{align*}
\sum_{n\leq x}n^{\alpha}\leq x^{\alpha+1}.
\end{align*}
\end{lemma}
\begin{proof}
From Euler's summation formula:
\begin{align*}
\sum_{n\leq x}n^{\alpha} & = \int_{1}^{x} t^{\alpha}\dt +\alpha\int_{1}^{x}t^{\alpha -1}(t-\lfloor t\rfloor )\dt +1 -(x-\lfloor x\rfloor )x^{\alpha},
\end{align*}
and
\begin{align*}
\alpha\int_{1}^{x}t^{\alpha -1}(t-\lfloor t\rfloor )\dt +1 -(x-\lfloor x\rfloor )x^{\alpha} < \alpha\int_{1}^{x}t^{\alpha -1}\dt +1 = x^{\alpha }.
\end{align*}
For the second part:
\begin{align*}
\sum_{n\leq x}n^{\alpha}\leq \sum_{n\leq x}x^{\alpha}=\lfloor x\rfloor x^{\alpha}\leq x^{\alpha +1}.
\end{align*}
\end{proof}

\begin{lemma}\label{LemFormulat}
Let us define recursively the following sequence:
\begin{align*}
\left\{\begin{array}{l}
t_{0}=0, \\
t_{k+1}=\frac{n-k}{n-k-1}\left(t_{k} +1\right).
\end{array}\right.
\end{align*}
Then for every $k\in\{0,\ldots ,n-1\}$:
\begin{align*}
t_{k} = \frac{kn-\frac{k(k-1)}{2}}{n-k} = \frac{k(2n-k+1)}{2(n-k)}.
\end{align*}
\end{lemma}
\begin{proof}
For every $k$ let us write $t_{k}=\frac{r_{k}}{n-k}$. Then from the recurrence defining $t_{k}$ we obtain:
\begin{align*}
\left\{\begin{array}{ll}
r_{0}=0,\\ 
r_{k}= r_{k-1}+n-k+1,
\end{array}\right.
\end{align*}
from which the result easily follows.
\end{proof}

\begin{lemma}\label{LemFormulau}
Let us define recursively the following sequence:
\begin{align*}
\left\{\begin{array}{l}
u_{1}=0, \\
u_{k+1}=\max\left\{\frac{n-k}{n-k-1}t_{k}, \frac{n-k}{n-k-1}(u_{k}+1)\right\}.
\end{array}\right.
\end{align*}
Then for every $k\in\{1,\ldots ,n-1\}$:
\begin{align*}
u_{k}=\frac{(k-1)(2n-k+2)}{2(n-k)}.
\end{align*}
\end{lemma}
\begin{proof}
It is enough to show that for every $k$:
\begin{align*}
t_{k}-u_{k}=\frac{n-k+1}{n-k}.
\end{align*}         
We use induction on $k$. For $k=1$ we have
\begin{align*}
t_{1}-u_{1}=\frac{n}{n-1}-0=\frac{n}{n-1}.
\end{align*}
If it is true for some $k-1$ then
\begin{align*}
t_{k}-u_{k} &= \frac{k(2n-k+1)}{2(n-k)}-\max\left\{\frac{n-k+1}{n-k}t_{k-1},\frac{n-k+1}{n-k}\left(t_{k-1}-\frac{1}{n-k+1}\right)\right\} \\
&= \frac{k(2n-k+1)}{2(n-k)}-\frac{n-k+1}{n-k}\cdot\frac{(k-1)(2n-k+2)}{2(n-k+1)}= \frac{n-k+1}{n-k},
\end{align*}
as desired.
\end{proof}

\begin{proof}[Proof of Theorem \ref{ThmIneqMain}]
We have:
\begin{align*}
\mathcal{S}_{\mathcal{B}}(A) = \sum_{A_{2}\leq B_{1}A_{1}^{2}}\ \ \sum_{A_{3}\leq B_{2}A_{2}^{3/2}}\ \ \cdots \sum_{A_{n}\leq B_{n-1}A_{n-1}^{n/(n-1)}} 1.
\end{align*}

Denote
\begin{align*}
\mathcal{T}_{k} := \sum_{A_{k}\leq B_{k-1}A_{k-1}^{k/(k-1)}}\ \ \sum_{A_{k+1}\leq B_{k}A_{k}^{(k+1)/k}}\ \ \cdots \sum_{A_{n}\leq B_{n-1}A_{n-1}^{n/(n-1)}} 1.
\end{align*}
In particular, $\mathcal{S}_{\mathcal{B}}(A)=\mathcal{T}_{2}$.

We will show that then for every $k$:
\begin{align}\label{EquTn-k}
\mathcal{T}_{n-k} = \prod_{j=1}^{k+1} \frac{B_{n-j}^{t_{j-1}+1}}{t_{j-1}+1} \cdot A_{n-k-1}^{t_{k+1}} +  O_{\leq 1}\left((k+1)\prod_{j=1}^{k}C_{n-j} \cdot A_{n-k-1}^{u_{k+1}} \right),
\end{align}
where $C_{n-j}:=\max\left\{ B_{n-j-1}^{t_{j}}, B_{n-j-1}^{u_{j}+1} \right\}$.

Indeed, for $k=0$ we have
\begin{align*}
\mathcal{T}_{n} = \sum_{A_{n}\leq B_{n-1}A_{n-1}^{n/(n-1)}} 1 = \left\lfloor B_{n-1}A_{n-1}^{n/(n-1)} \right\rfloor = B_{n-1}A_{n-1}^{n/(n-1)} + O_{\leq 1}(1).
\end{align*}
Assume it holds for some $k-1$ and denote for simplicity:
\begin{align*}
\varphi_{k} := \prod_{j=1}^{k+1} \frac{B_{n-j}^{t_{j-1}+1}}{t_{j-1}+1} \hspace{1cm} \textrm{ and } \hspace{1cm} \psi_{k}:=(k+1)\prod_{j=1}^{k}C_{n-j}.
\end{align*} 
Then for $k$ we get
\begin{align*}
\mathcal{T}_{n-k} & = \sum_{A_{n-k}\leq B_{n-k-1}A_{n-k-1}^{(n-k)/(n-k-1)}}\mathcal{T}_{n-k+1} \\
& = \varphi_{k-1}\sum_{A_{n-k}\leq B_{n-k-1}A_{n-k-1}^{(n-k)/(n-k-1)}} A_{n-k}^{t_{k}} + O_{\leq 1}\left(\psi_{k-1}\sum_{A_{n-k}\leq B_{n-k-1}A_{n-k-1}^{(n-k)/(n-k-1)}}A_{n-k}^{u_{k}} \right) \\
& = \varphi_{k-1}\cdot \frac{B_{n-k-1}^{t_{k}+1}}{t_{k}+1 }\cdot A_{n-k-1}^{\frac{n-k}{n-k-1}(t_{k}+1)} \\
& \ \ \ \ + O_{\leq 1}\left(\varphi_{k-1}B_{n-k-1}^{t_{k}}                                             A_{n-k-1}^{\frac{n-k}{n-k-1}t_{k}}\right)  + O_{\leq 1}\left( \psi_{k-1} B_{n-k-1}^{u_{k}+1} A_{n-k-1}^{\frac{n-k}{n-k-1}(u_{k}+1)} \right) \\
 & = \varphi_{k}A_{n-k-1}^{\frac{n-k}{n-k-1}(t_{k}+1)} + O_{\leq 1}\left(\varphi_{k-1}C_{n-k} A_{n-k-1}^{u_{k+1}}\right) + O_{\leq 1}\left( \psi_{k-1}C_{n-k}  A_{n-k-1}^{u_{k+1}} \right) \\
& = \varphi_{k}A_{n-k-1}^{t_{k+1}} + O_{\leq 1}\left( \psi_{k}A_{n-k-1}^{u_{k+1}} \right).
\end{align*}

By specifying to $k=n-2$ we get:
\begin{align*}
\mathcal{S}_{n}(A) & = \Phi_{n} A^{t_{n-1}} + O_{\leq 1}\left(\Psi_{n}A^{u_{n-1}}\right),
\end{align*}
where  
\begin{align*}
t_{n-1}& = \frac{(n-1)(n+2)}{2}, \\
\Phi_{n} & := \varphi_{n-2} = \prod_{j=1}^{n-1} \frac{B_{n-j}^{t_{j-1}+1}}{t_{j-1}+1}, \\
\Psi_{n} & := \psi_{n-2} = (n-1)\prod_{j=1}^{n-2}C_{n-j}.
\end{align*}
The result follows because $t_{n-1}-u_{n-1} =2$. Indeed, we can turn the numbers $\Phi_{\mathcal{B}}$ and $\Psi_{\mathcal{B}}$ defined above into a more explicit form using the following equalities:
\begin{align*}
\prod_{j=1}^{n-1}\frac{1}{t_{j-1}+1} & = \prod_{j=1}^{n-1}\frac{n-j+1}{n-j}\prod_{j=1}^{n-1}\frac{1}{t_{j}} = n\prod_{j=1}^{n-1}\frac{2(n-j)}{j(2n-j+1)} = 2^{n-1}n \frac{(n+1)!}{(2n)!}, \\
\prod_{j=1}^{n-1}B_{n-j}^{t_{j-1}+1} &= \prod_{k=1}^{n-1}B_{k}^{t_{n-k-1}+1}=\prod_{k=1}^{n-1}B_{k}^{\frac{k}{k+1}t_{n-k}}=\prod_{k=1}^{n-1}B_{k}^{\frac{(n-k)(n+k+1)}{2(k+1)}}, \\
\prod_{j=1}^{n-2}C_{n-j} &= \prod_{j=1}^{n-2}\max\left\{B_{n-j-1}^{t_{j}},B_{n-j-1}^{u_{j}+1}\right\}=\prod_{j=1}^{n-2}B_{n-j-1}^{t_{j}}\cdot\prod_{j=1}^{n-2}\max\left\{1,B_{n-j-1}^{u_{j}+1-t_{j}}\right\} \\
&= \prod_{j=1}^{n-2}B_{n-j-1}^{t_{j}}\cdot\prod_{j=1}^{n-2}\max\left\{1,B_{n-j-1}^{-\frac{1}{n-j}}\right\}=\prod_{k=1}^{n-2}B_{k}^{t_{n-k-1}}\cdot\prod_{k=1}^{n-2}\max\left\{1,B_{k}^{-\frac{1}{k+1}}\right\} \\
&= \prod_{k=1}^{n-2}B_{k}^{\frac{(n-k-1)(n+k+2)}{2(k+1)}}\cdot\prod_{k=1}^{n-2}\max\left\{1,B_{k}^{-\frac{1}{k+1}}\right\}.
\end{align*}
This finishes the proof.
\end{proof}

In practice, we will use conditions of the form $(A_{k+1}/D_{k+1})^{1/(k+1)}\leq (A_{k}/D_{k})^{1/k}$ (or similar). By using such numbers $D_{k}$ we can simplify some of the products that occur in the formulas for $\Phi_{\mathcal{B}}$ and $\Psi_{\mathcal{B}}$:

\begin{lemma}\label{LemSimplification}
Assume that for every $k\geq 1$
\begin{align*}
E_{k}=\frac{D_{k+1}}{D_{k}^{\frac{k+1}{k}}}.
\end{align*}
Then
\begin{align*}
\prod_{k=1}^{n-1}E_{k}^{\frac{(n-k)(n-k+1)}{2(k+1)}} &= \frac{1}{D_{1}^{\frac{(n-1)(n+2)}{2}}}\cdot\prod_{k=2}^{n}D_{k},  \\
\prod_{k=1}^{n-2}E_{k}^{\frac{(n-k-1)(n+k+2)}{2(k+1)}}\cdot\prod_{k=1}^{n-2}\max\left\{1,E_{k}^{-\frac{1}{k+1}}\right\} &= \frac{D_{n-1}^{\frac{n}{n-1}}}{D_{1}^{\frac{(n-1)(n+2)}{2}}}\cdot\prod_{k=1}^{n-2}D_{k}\cdot\prod_{k=1}^{n-2}\max\left\{D_{k}^{\frac{1}{k}},D_{k+1}^{\frac{1}{k+1}}\right\}. 
\end{align*}
\end{lemma}
\begin{proof}
The quantity on the left in the first equality is in fact $\prod_{j=1}^{n-1}E_{n-j}^{t_{j-1}+1}$. Hence,
\begin{align*}
\prod_{j=1}^{n-1}E_{n-j}^{t_{j-1}+1} &= \prod_{j=1}^{n-1}\frac{D_{n-j+1}^{t_{j-1}+1}}{D_{n-j}^{\frac{n-j+1}{n-j}(t_{j-1}+1)}}=\prod_{j=1}^{n-1}\frac{D_{n-j+1}^{t_{j-1}+1}}{D_{n-j}^{t_{j}}} \\
&= \prod_{j=1}^{n-1}D_{n-j+1}\prod_{j=1}^{n-1}\frac{D_{n-(j-1)}^{t_{j-1}}}{D_{n-j}^{t_{j}}}=\prod_{j=2}^{n}D_{j}\cdot\frac{1}{D_{1}^{t_{n-1}}}.
\end{align*}

In the second equality, we deal with:
\begin{align*}
\prod_{j=1}^{n-2}\max\left\{E_{n-j-1}^{t_{j}},E_{n-j-1}^{u_{j}+1}\right\} &= \prod_{j=1}^{n-2} E_{n-j-1}^{t_{j}} \prod_{j=1}^{n-2} \max\left\{1,E_{n- j-1}^{\frac{1}{n-j}}\right\}.
\end{align*} 
For the first of the products above we can apply a similar reasoning, as in the proof of the first part. We get
\begin{align*}
\prod_{j=1}^{n-2} E_{n-j-1}^{t_{j}} &= \frac{\prod_{j=1}^{n-2} E_{n-j-1}^{t_{j}+1}}{\prod_{j=1}^{n-2} E_{n-j-1}} = \frac{\prod_{j=1}^{n-2}D_{n-j}^{t_{j}+1}/D_{n-j-1}^{\frac{n-j}{n-j-1}(t_{j}+1)}}{\prod_{j=1}^{n-2}D_{n-j}/D_{n-j-1}^{\frac{n-j}{n-j-1}}} = \prod_{j=1}^{n-2}D_{n-j-1}^{\frac{n-j}{n-j-1}}\prod_{j =1}^{n-2}\frac{D_{n-j}^{t_{j}}}{D_{n-j-1}^{t_{j+1}}} \\
&= \prod_{j=1}^{n-2}D_{n-j-1}^{\frac{n-j}{n-j-1}}\cdot\frac{D_{n-1}^{t_{1}}}{D_{1}^{t_{n-1}}}.
\end{align*}
For the second product:
\begin{align*}
\prod_{j=1}^{n-2} \max\left\{1,E_{n-j-1}^{\frac{1}{n-j}}\right\} &= \prod_{j=1}^{n-2}\max\left\{1,\frac{D_{n-j}^{\frac{1}{n-j}}}{D_{n-j-1}^{\frac{1}{n-j-1}}}\right\} =\frac{\prod_{j=1}^{n-2}\max\left\{D_{n-j-1}^{\frac{1}{n-j-1}},D_{n-j}^{\frac{1}{n-j}}\right\}}{\prod_{j=1}^{n-2}D_{n-j-1}^{\frac{1}{n-j-1}}}.
\end{align*}
Hence, finally
\begin{align*}
\prod_{j=1}^{n-2}\max\left\{E_{n-j-1}^{t_{j}},E_{n-j-1}^{u_{j}+1}\right\} &= \frac{D_{n-1}^{t_{1}}}{D_{1}^{t_{n-1}}}\prod_{j=1}^{n-2}D_{n-j-1}\prod_{j=1}^{n-2}\max\left\{D_{n-j-1}^{\frac{1}{n-j-1}},D_{n-j}^{\frac{1}{n-j}}\right\} \\
&= \frac{D_{n-1}^{t_{1}}}{D_{1}^{t_{n-1}}}\prod_{k=1}^{n-2}D_{k}\prod_{k=1}^{n-2}\max\left\{D_{k}^{\frac{1}{k}},D_{k+1}^{\frac{1}{k+1}}\right\}.\qedhere
\end{align*}
\end{proof}

\section{Proof of Theorem \ref{ThmIneq1}}

\begin{lemma}\label{LemAsympt}
When $n\to\infty$ then the following equalities are true:
\begin{enumerate}
\item $\sum_{k=1}^{n}\log k=n\log n-n+O(\log n)$,
\item $\sum_{k=1}^{n}k\log k=\frac{1}{2}n^{2}\log n-\frac{1}{4}n^{2}+O(n\log n)$,
\item $\sum_{k=1}^{n}\log\binom{n}{k}=\frac{1}{2}n^{2}+O(n\log n)$,
\item $\sum_{k=1}^{n}\frac{1}{k}\log\binom{n}{k}=O(n)$.
\end{enumerate}
\end{lemma}
\begin{proof}
Equality (1) is well-known. In order to prove (2), we get from Euler's summation formula:
\begin{align*}
\sum_{k=1}^{n}k\log k &= \int_{1}^{n}t\log t\dt +O(n\log n)=\frac{1}{2}n^{2}\log n-\frac{1}{4}n^{2} + O(n\log n).
\end{align*}

In the proof of (3) we will use the previous two identities:
\begin{align*}
\sum_{k=1}^{n}\log\binom{n}{k} &= n\log(n!)-2\sum_{k=1}^{n}\log(j!)=n\sum_{k=1}^{n}\log k - 2\sum_{j=1}^{n}\sum_{k=1}^{j}\log k  \\
&= n \sum_{k=1}^{n}\log k - 2\sum_{k=1}^{n}(n-k+1)\log k = 2\sum_{k=1}^{n}k\log k - n\sum_{k=1}^{n}\log k - 2\sum_{k=1}^{n}\log k \\
&=  \left(n^{2}\log n-\frac{1}{2}n^{2}+O(n\log n)\right) - \left(n^{2}\log n-n^{2}+O(n\log n)\right) + O(n\log n) \\
&= \frac{1}{2}n^{2}+O(n\log n).
\end{align*}

Equality (4) follows from (3) by partial summation.
\end{proof}

\begin{proof}[Proof of Theorem \ref{ThmIneq1}]
Inequality $s_{n,k}(A_{k})\geq s_{n,k+1}(A_{k+1})$ is equivalent to
\begin{align*}
A_{k+1}\leq \frac{\binom{n}{k+1}}{\binom{n}{k}^{\frac{k+1}{k}}}A_{k}^{\frac{k+1}{k}}.
\end{align*}
Therefore, the first part follows directly from Theorem \ref{ThmIneqMain} and Lemma \ref{LemSimplification} used with $B_{k}=D_{k+1}/D_{k}^{\frac{k+1}{k}}$ and $D_{k}=\binom{n}{k}$. Indeed, we only need to show that
\begin{align*}
\max\left\{\binom{n}{k}^{\frac{1}{k}},\binom{n}{k+1}^{\frac{1}{k+1}}\right\}=\binom{n}{k}^{\frac{1}{k}}.
\end{align*}
This can be seen by writing:
\begin{align*}
\binom{n}{k}^{\frac{1}{k}-\frac{1}{k+1}}= \left(\frac{n(n-1)\cdots (n-k+1)}{k!}\right)^{\frac{1}{k(k+1)}}>\left(\frac{(n-k)^{k}}{(k+1)^{k}}\right)^{\frac{1}{k(k+1)}}=\left(\frac{n-k}{k+1}\right)^{\frac{1}{k+1}},
\end{align*}
and then multiplying both sides by $\binom{n}{k}^{\frac{1}{k+1}}$ and using the property $\frac{n-k}{k+1}\binom{n}{k}=\binom{n}{k+1}$.

Let us move on to the second part of the statement. By the first part we have
\begin{align*}
\log\mathcal{S}_{n}(A) &= \left(\frac{1}{2}+o_{n\to\infty}(1)\right)n^{2}\log A + \log\Phi_{n} + \log\left(1+O_{<1}\left(\frac{\Psi_{n}}{\Phi_{n}A^{2}}\right)\right).
\end{align*}
Then, thanks to Lemma \ref{LemAsympt}, we have
\begin{align*}
\log\Phi_{n} &= \log\left(\prod_{j=2}^{n}\binom{n}{j}\right)-\left(\frac{1}{2}+o_{n\to\infty}(1)\right)n^{2}\log n+o_{n\to\infty}(n^{2}) \\
&= \left(\frac{1}{2}+o_{n\to\infty}(1)\right)n^{2}-\left(\frac{1}{2}+o_{n\to\infty}(1)\right)n^{2}\log n = \left(\frac{1}{2}+o_{n\to\infty}(1)\right)n^{2}\log\left(\frac{e}{n}\right).
\end{align*}
Moreover,
\begin{align*}
\frac{\Psi_{n}}{\Phi_{n}} &= \frac{(n-1)n^{1+\frac{n}{n-1}}}{2^{n-1}}\cdot \frac{(2n)!}{(n+1)!}\prod_{k=1}^{n-2}\binom{n}{k}^{\frac{1}{k}},
\end{align*}
so
\begin{align*}
\log\left(\frac{\Psi_{n}}{\Phi_{n}A^{2}}\right) &\leq \log\left(\frac{\Psi_{n}}{\Phi_{n}}\right) = \sum_{k=1}^{n}\frac{\log\binom{n}{k}}{k} + o_{n\to\infty}(n^{2})=o_{n\to\infty}(n^{2}),
\end{align*}
and the result follows.
\end{proof}

\section{Cubic polynomials}

\begin{proof}[Proof of Theorem \ref{ThmCubic}.]
Suppose $F(X)=X^{3}-A_{1}X^{2}+A_{2}X-A_{3}$ has only real roots. In particular, its discriminant $\Delta_{f}$ is nonnegative. But
\begin{align*}
    \Delta_{f} = -27A_{3}^2 + (-4A_{1}^3 + 18A_{1}A_{2})A_{3} + (A_{1}^2A_{2}^2 - 4A_{2}^3),
\end{align*}
so, it is possible only if the discriminant of the above expression (as a polynomial in $A_{3}$) is nonnegative, that is,    $16(A_{1}^{2}-3A_{2})^{3}\geq 0$, and
\begin{align}\label{IneqCubicA3}
    \frac{1}{27}\left(9A_{1}A_{2}-2A_{1}^{3}-2(A_{1}^{2}-3A_{2})^{3/2}\right) \leq A_{3} \leq \frac{1}{27}\left(9A_{1}A_{2}-2A_{1}^{3}+2(A_{1}^{2}-3A_{2})^{3/2}\right).
\end{align}
Let $G_{-}(A_{2})$ and $G_{+}(A_{2})$ denote the expressions on the left and on the right-hand side of \eqref{IneqCubicA3}, respectively. 

We require the following:

\begin{lemma}\label{LemRobinson}
    Let $k\geq 2$, and $p(X)$ be a monic polynomial of degree $k-1$, with real zeros $\beta_{1}>\beta_{2}>\cdots >\beta_{k-1}>0$. Let $P(X)=k\int_{0}^{X}p(t)\dt$, monic of degree $k$. Then $P(X)-c$ has all zeros real and positive if and only if $(-1)^{k}c<0$ and
    \begin{align*}
        \max_{i=1,\ldots ,\lfloor k/2\rfloor} P(\beta_{2i-1}) \leq c \leq \min_{i=1,\ldots ,\lfloor (k-1)/2\rfloor} P(\beta_{2i}).
    \end{align*}
\end{lemma}
\begin{proof}
    This is in fact Robinson's method. In this form, it can be found in Smyth's \cite{Smyth2}.
\end{proof}

Let $f(X)=X^{3}-A_{1}X^{2}+A_{2}X-A_{3}$, where $A_{1},A_{2},A_{3}\geq 0$. If the zeros of $f$ are real and positive, so are the zeros of $\frac{1}{3}f'(X)=X^{2}-\frac{2}{3}A_{1}X+\frac{1}{3}A_{2}$. The latter holds if and only if $A_{1}\geq 3A_{2}$. Then we can compute the zeros of $\frac{1}{3}f'(x)$ and apply Lemma \ref{LemRobinson} with $p(X)=\frac{1}{3}f'(X)$, $P(X):=f(X)-f(0)$ and $c=A_{3}$. We find that $f(X)$ has all its zeros real and positive if and only if both $A_{3}> 0$ and \eqref{IneqCubicA3}.

The calculations show that the expression on the very left of \eqref{IneqCubicA3} is positive only if $A_2> A_1^{2}/4$. Therefore,
\begin{align*}
   \#\mathcal{P}_{3}^{+}(A) = \sum_{A_{2}\leq A^{2}/4}\ \sum_{0< A_{3}\leq G_{+}(A_{2})} 1 + \sum_{A^{2}/4<A_{2}\leq A^{2}/3}\ \sum_{G_{-}(A_{2})\leq A_{3}\leq G_{+}(A_{2})} 1,
\end{align*}
where $G_{-}(A_{2})$ and $G_{+}(A_{2})$ are the same as before.

The first range can be computed using Lemma \ref{LemEuler}.
\begin{align*}
    \sum_{A_{2}\leq A^{2}/4}\ \sum_{0\leq A_{3}\leq G_{+}(A_{2})} 1 &= \sum_{A_{2}\leq A^{2}/4} \lfloor G_{+}(A_{2})\rfloor \\
    &= \frac{A}{3}\sum_{A_{2}\leq A^{2}/4}A_{2} - \frac{2A^{3}}{27}\sum_{A_{2}\leq A^{2}/4}1 + \frac{2}{27}\sum_{0<A_{2}\leq A^{2}/4}(A^{2}-3A_{2})^{3/2} \\
    & \ \ \ + \frac{2}{27}A^{3} + O_{\leq 1}\left(\frac{1}{4}A^{2}\right) \\
    &= \frac{A}{6}\left\lfloor\frac{A^{2}}{4}\right\rfloor\left( \left\lfloor\frac{A^{2}}{4}\right\rfloor+1 \right) - \frac{2A^{3}}{27} \left\lfloor\frac{A^{2}}{4}\right\rfloor + \frac{2}{27}\int_{0}^{A^{2}/4} (A^{2}-3t)^{3/2} \dt \\
    & \ \ \ + O_{\leq 1}\left(\frac{1}{3} \int_{0}^{A^{2}/4} (A^{2}-3t)^{1/2} \dt \right) + O_{\leq 1}\left(\frac{1}{8}A^{3}\right) + \frac{2}{27}A^{3} + O_{\leq 1}\left(\frac{1}{4}A^{2}\right) \\
    &= \frac{1}{96}A^{5} + O_{\leq 1}\left(\frac{1}{8}A^{3}+\frac{1}{3}A\right) - \frac{1}{54}A^{5} +O_{\leq 1}\left(\frac{2}{27}A^{3}\right) \\
    & \ \ \ + \frac{31}{3240}A^{5} + O_{\leq 1}\left(\frac{7}{108}A^{3}+\frac{5}{54}A^{2}\right) + \frac{2}{27}A^{3} + O_{\leq 1}\left(\frac{1}{4}A^{2}\right) \\
    &= \frac{19}{12960}A^{5} + O_{\leq 1}\left(\frac{73}{216}A^{3} + \frac{37}{108}A^{2} +\frac{1}{3}A\right).
\end{align*}

Similarly, for the second range, we get
\begin{align*}
    \sum_{A^{2}/4<A_{2}\leq A^{2}/3}\ \sum_{G_{-}(A_{2})\leq A_{3}\leq G_{+}(A_{2})} 1 & = \sum_{A^{2}/4<A_{2}\leq A^{2}/3} \left(\left\lfloor G_{+}(A_{2})\right\rfloor - \left\lceil G_{-}(A_{2})\right\rceil\right) \\
    & = \frac{4}{27} \sum_{A^{2}/4<A_{2}\leq A^{2}/3} (A^{2}-3A_{2})^{3/2} + O_{\leq 1}\left(\frac{2}{3}A^{2}\right) \\
    & = \frac{4}{27}\int_{A^{2}/4}^{A^{2}/3} (A^{2}-3t)^{3/2}\dt + O_{\leq 1}\left(\frac{2}{3}\int_{A^{2}/4}^{A^{2}/3} (A^{2}-3t)^{1/2}\dt\right) \\
    & \ \ \ + O_{\leq 1}\left(\frac{1}{54}A^{3}\right)+ O_{\leq 1}\left(\frac{2}{3}A^{2}\right) \\
    & = \frac{1}{1620}A^{5} + O_{\leq 1}\left(\frac{1}{27}A^{3}+\frac{2}{3}A^{2}\right).
\end{align*}

By summing both estimates, we finally obtain:
\begin{align*}
   \#\mathcal{P}_{3}^{+}(A) & = \frac{1}{480}A^{5} + O_{\leq 1}\left(\frac{3}{8}A^{3} + \frac{109}{108}A^{2} + \frac{1}{3}A \right) \\
   & = \frac{1}{480}A^{5} + O_{\leq 1}\left(2A^{3}\right),
\end{align*}
as desired.
\end{proof}

At the end of this section we present two interesting results that use reasoning similar to that in the proof of Theorem \ref{ThmCubic}. In Corollary \ref{Cor3Coeffs} we will actually need the assumption that we count polynomials with all roots real and nonnegative. Therefore, let us introduce the following notation. For every string of positive real numbers $\underline{\alpha}:=(\alpha_{1},\ldots ,\alpha_{n})$, let
\begin{align*}
    \mathcal{P}_{n}^{0+}(A;\underline{\alpha}) := \Bigg\{ (A_{1},\ldots ,A_{n})\in\mathbb{Z}^{n}\ \Bigg|\ & A_{1}=A,\ X^{n}+\sum_{i=1}^{n} (-1)^{i}\alpha_{i}A_{i}X^{n-i} \\ 
    & \textit{ has only nonnegative real zeros}  \Bigg\}.
\end{align*}

Note that $\mathcal{P}_{3}^{0+}(A):=\mathcal{P}_{3}^{0+}(A;(1,1,1))$ counts the number of cubic polynomials with trace $A$ and all real and nonnegative roots.

We can repeat the computations from the proof of Theorem \ref{ThmCubic} and obtain the following generalisation:

\begin{theorem}\label{ThemCubicabc}
    For every triple $(\alpha,\beta,\gamma)$ of positive real numbers we have
    \begin{align*}
        \#\mathcal{P}_{3}^{+}(A;(\alpha,\beta,\gamma)) & = \frac{\alpha^{5}}{480\beta\gamma} A^{5} + O_{\leq 1}\left(\frac{\alpha^{3}}{\gamma}A^{3} + \left(\frac{1}{\beta} + \frac{1}{\gamma}\right)\alpha^{2}A^{2} + \frac{\alpha\beta}{\gamma}A \right), \\
        \#\mathcal{P}_{3}^{0+}(A;(\alpha,\beta,\gamma)) & = \frac{\alpha^{5}}{480\beta\gamma} A^{5} + O_{\leq 1}\left(\frac{\alpha^{3}}{\gamma}A^{3} + \left(\frac{2}{\beta} + \frac{1}{\gamma}\right)\alpha^{2}A^{2} + \frac{\alpha\beta}{\gamma}A \right).
    \end{align*}
\end{theorem}
\begin{proof}
    We follow the argument from the proof of Theorem \ref{ThmCubic}. We obtain the following: 
    \begin{align*}
        \#\mathcal{P}_{3}^{+}(A;(\alpha,\beta,\gamma)) = \sum_{A_{2}\leq \alpha^{2}A^{2}/(4\beta)}\ \sum_{0< A_{3}\leq G_{+}(\beta A_{2})/\gamma} 1 + \sum_{\alpha^{2}A^{2}/(4\beta)<A_{2}\leq \alpha^{2}A^{2}/(3\beta)}\ \sum_{G_{-}(\beta A_{2})/\gamma\leq A_{3}\leq G_{+}(\beta A_{2})/\gamma} 1,
    \end{align*}
    where in the definitions of $G_{-}(A_{2})$ and $G_{+}(A_{2})$ we replace $A_{1}$ by $\alpha A_{1}$.
    
    The remaining computations are analogous to those previously performed and are omitted for brevity. In fact, one obtains a more precise estimate of the error term:
    \begin{align*}
        \#\mathcal{P}_{3}^{+}(A;(\alpha,\beta,\gamma)) = \frac{\alpha^{5}}{480\beta\gamma} A^{5} + O_{\leq 1}\left(\frac{3\alpha^{3}}{8\gamma}A^{3} + \left(\frac{11\alpha^{2}}{12\beta} + \frac{5\alpha^{2}}{54\gamma}\right)A^{2} + \frac{\alpha\beta}{3\gamma}A \right),
    \end{align*}
    which directly implies (by putting $\alpha=\beta=\gamma=1$) the error term $O_{\leq}\left(\frac{3}{8}A^{3} + \frac{109}{108}A^{2} + \frac{1}{3}A\right)$ proved at the end of the proof of Theorem \ref{ThmCubic}.

    In order to obtain the result about $\mathcal{P}_{3}^{0+}(A;(\alpha,\beta,\gamma))$, observe that the reasoning from the proof of Theorem \ref{ThmCubic} in this case yields
    \begin{align*}
        \#\mathcal{P}_{3}^{0+}(A;(\alpha,\beta,\gamma)) & = \sum_{A_{2}\leq \alpha^{2}A^{2}/(4\beta)}\ \sum_{0\leq A_{3}\leq G_{+}(\beta A_{2})/\gamma} 1 + \sum_{\alpha^{2}A^{2}/(4\beta)<A_{2}\leq \alpha^{2}A^{2}/(3\beta)}\ \sum_{G_{-}(\beta A_{2})/\gamma\leq A_{3}\leq G_{+}(\beta A_{2})/\gamma} 1 \\
        & = \#\mathcal{P}_{3}^{+}(A;(\alpha,\beta,\gamma)) + \sum_{A_{2}\leq \alpha^{2}A^{2}/(4\beta)} 1 = \#\mathcal{P}_{3}^{+}(A;(\alpha,\beta,\gamma)) + O_{\leq 1}\left(\frac{\alpha^{2}}{\beta}A^{2}\right).
    \end{align*}
    Using the result about $\#\mathcal{P}_{3}^{+}(A;(\alpha,\beta,\gamma))$, we therefore obtain the estimate:
    \begin{align*}
        \#\mathcal{P}_{3}^{0+}(A;(\alpha,\beta,\gamma)) = \frac{\alpha^{5}}{480\beta\gamma} A^{5} + O_{\leq 1}\left(\frac{3\alpha^{3}}{8\gamma}A^{3} + \left(\frac{7\alpha^{2}}{6\beta} + \frac{5\alpha^{2}}{54\gamma}\right)A^{2} + \frac{\alpha\beta}{3\gamma}A \right).
    \end{align*}
    The result follows by further bounding the error term.
\end{proof}

\begin{cor}\label{Cor3Coeffs}
    For every $n\geq 4$:
    \begin{align*}
        \# & \left\{(A_{1},A_{2},A_{3})\Bigg|\ \exists_{A_{4},\ldots ,A_{n}}\ X^{n} +\sum_{i=1}^{n} (-1)^{i}A_{i}X^{n-i} \in\mathcal{P}_{n}^{0+}(A) \right\} \\
        & = \frac{27}{640}\left(1-\frac{1}{n}\right)^{2}\left(1-\frac{2}{n}\right) A^{5} + O_{\leq 1}\left(\frac{9}{2}A^{3} + \frac{3}{2}nA^{2} + 3A\right).
    \end{align*}
    In the case of $\mathcal{P}_{n}^{+}(A)$ instead of $\mathcal{P}_{n}^{0+}(A)$, we only have the upper bound.
\end{cor}
\begin{proof}
    Let
    \begin{align*}
        f(X)=X^{n}-A_{1}X^{n-1}+A_{2}X^{n-2}-\cdots+(-1)^{n-1}A_{n-1}X+(-1)^{n}A_{n}
    \end{align*}
    has only positive roots. The idea is to consider its derivatives and apply Theorem \ref{ThemCubicabc}. Indeed, observe that if $f(X)$ has only positive roots, then the same is true for its derivatives. Hence, for every $k$ we have that the following polynomial has only positive roots:
    \begin{align*}
        \frac{(n-k)!}{n!}f^{(k)}(X)=X^{n-k}-\frac{n-k}{n}A_{1}X^{n-1-k}+\frac{(n-k)(n-k-1)}{n(n-1)}A_{2}X^{n-2-k}+\cdots 
    \end{align*}    
    where in general, the coefficient of $X^{j-k}$ is $\frac{(n-k)!}{n!}\cdot\frac{j!}{(j-k)!}\cdot A_{j}$. In particular, for $k=n-3$ we get
    \begin{align*}
        \frac{6}{n!}f^{(n-3)}(X) = X^{3} - \frac{3}{n}A_{1}X^{2} + \frac{6}{n(n-1)}A_{2}X - \frac{6}{n(n-1)(n-2)}A_{3}.
    \end{align*}
    On the other hand, if $A_{1},A_{2},A_{3}$ are such that the polynomial $g(X):=X^{3} - \frac{3}{n}A_{1}X^{2} + \frac{6}{n(n-1)}A_{2}X - \frac{6}{n(n-1)(n-2)}A_{3}$ has only nonnegative roots, then $X^{n-3}g(X)$ is a polynomial of degree $n$ with all its roots real and nonnegative. Therefore, the quantity in the statement is equal to
    \begin{align*}
        \#\Bigg\{(A_{1},A_{2},A_{3})\in \mathbb{Z}^{3}\ \Bigg|\ & X^{3} - \frac{3}{n}A_{1}X^{2} + \frac{6}{n(n-1)}A_{2}X - \frac{6}{n(n-1)(n-2)}A_{3} \\
        & \textit{ has all the roots real and nonnegative},\ A_{1}=A \Bigg\}.
    \end{align*}
    The latter can be directly computed by Theorem \ref{ThemCubicabc} used with $\alpha=\frac{3}{n}$, $\beta = \frac{6}{n(n-1)}$, $\gamma = \frac{6}{n(n-1)(n-2)}$.
\end{proof}

\section{Square-free values of quadratic polynomials}

In this section, we apply the techniques in \cite{Murty}. Let $f(X)=AX^2+BX+C$, $\Delta_f=B^2-4AC$ be its discriminant, and let $\rho_f(d)$ denote the number of solutions to $f(X)\equiv 0\pmod{d}$. 

\begin{enumerate}
    \item Assume $p$ is an odd prime and $p\nmid A$. Then

    \[\rho_f(p)=\begin{cases}
0 &\;\mbox{if } \left(\dfrac{\Delta_f}{p}\right)=-1,\\
     1 &\;\mbox{if } p\mid \Delta_f,\\
    2 &\;\mbox{if }\left(\dfrac{\Delta_f}{p}\right)=1,
\end{cases}\]
where $\left(\frac{\cdot}{\cdot}\right)$ is the Legendre symbol. Now, lifting the solutions modulo $p^2$: 

Clearly, there are no solutions if $\rho_f(p)=0$. If $\left(\dfrac{\Delta_f}{p}\right)=1$, then the two simple roots modulo $p$ can be lifted by Hensel's lemma, so that $\rho_f(p^2)=2$. Finally, we distinguish two cases for $p\mid \Delta_f$:
\begin{itemize}
    \item If $\Delta_f\not\equiv 0\pmod{p^2}$, then the double root $X$ modulo $p$ is the root of the derivative $2AX+B\equiv0\pmod{p}$. Substituting $ -\dfrac{B}{2A}$ for $X$ in
    \[AX^2+BX+C=-\frac{\Delta_f}{4A}\not\equiv 0\pmod{p^2},\] 
    hence, $\rho_f(p^2)=0$.
    \item If $p^2\mid \Delta_f$, then by an argument similar to that above, every lift of $X$ is a root modulo $p^2$, so $\rho_f(p^2)=p$.
\end{itemize}
In summary, when $p\nmid A$ and $p$ is odd, then
\[\rho_f(p^2)=\begin{cases}
 0 &\;\mbox{if } \left(\dfrac{\Delta_f}{p}\right)=-1\mbox{ or }p\mid\mid \Delta_f,\\
 2 &\;\mbox{if } \left(\dfrac{\Delta_f}{p}\right)=1,\\
    p &\;\mbox{if }p^2\mid \Delta_f.
\end{cases}\]
\item Assume that $p$ is an odd prime number, and $p\mid A$. Then we have a linear (or a constant) polynomial modulo $p$, and clearly
\[\rho_f(p)=\begin{cases}
0 &\;\mbox{if } p\mid B \mbox{ and }p\nmid{C},\\
     1 &\;\mbox{if }p\nmid B,\\
    p &\;\mbox{if }p\mid \gcd(B,C).
\end{cases}\]
For lifts modulo $p^2$: 
\begin{itemize}
    \item If $p\nmid \Delta_f$, then $p\nmid B$. Hence, by Hensel's lemma we have a unique solution, so $\rho_f(p^2)=1$.
    \item If $p\mid \Delta_f$ but $p^2\nmid \Delta_f$, then from the discriminant we get $p\mid B$, but $p \nmid C$. Therefore, $\rho_f(p^2)=0$ (since we do not have a solution even modulo $p$).
    \item  If $p\mid \gcd(A,B,C)$, then we can consider $g=f/p$ and $p\cdot\rho_g(p)=\rho_f(p^2)$, for $\rho_g(p)\in \{0,1,2,p\}$.
\end{itemize}
In summary, for odd prime $p$ such that $p\mid A$: 
\[\rho_f(p^2)=\begin{cases}
0 &\;\mbox{if } p\mid\mid \Delta_f \mbox{ or }\rho_f(p)=0,\\
     1 &\;\mbox{if } p\nmid \Delta_f,\\
    p\cdot\rho_{f/p}(p)  &\;\mbox{if }p\mid \gcd(B,C).
\end{cases}\]
\item If $p=2$, then we have $\rho_f(4)=0,1,2$ or $4$.
\end{enumerate}

Since $\rho_{f}(d)$ is multiplicative, these considerations immediately imply the following result similar to Lemma 2 in \cite{Murty}:

\begin{lemma}\label{Lemrho(d2)}
    Let $f=AX^2+BX+C\in\Z[X]$ such that $\Delta_f\not=0$, and let $d$ be a square-free odd integer such that $\gcd(A,d)=1$. Then $\rho_f(d^2)\le 2^{\omega(d)}\mathrm{rad}(\Delta_f)$, where $\omega(d)$ denotes the number of distinct prime factors of $d$, and the radical of $\Delta_f$ is denoted by $\mathrm{rad}(\Delta_f)$.
\end{lemma}
If the discriminant above is zero, then $f$ is a square of a linear polynomial, and $\rho_f(d^2)\le d$.

Denote 
\begin{align*}
    S(x,y;f) & := \#\big\{ x<n\leq y\ \big|\ f(n) \textrm{ is square-free}  \big\}, \\
    N_{p}(x,y;f) & := \#\big\{ x<n\leq y\ |\ p^{2} \textrm{ divides } f(n) \big\} .
\end{align*}

\begin{theorem}\label{ThmSquarefreeQuadratic}
    Let $f(X)=AX^{2}+BX+C$ satisfy $\Delta_{f}>0$. If $x<y$ and $z<y$, then
    \begin{align*}
        S(x,y;f) & \geq (y-x) \prod_{\substack{p\leq z \\ p\mid \Delta_{f}}} \left(1 - \frac{1}{p}\right)  \prod_{\substack{p\leq z \\ p\mid\gcd (A,B,C)}} \left(1 - \frac{\rho_{f} (p^{2})}{p^{2}}\right)  \prod_{p\in\mathbb{P}} \left(1 - \frac{2}{p^{2}}\right)  \\
        & \ \ \ \  - \rad(\Delta_{f})\cdot 3^{\pi (z)} - (y-x)\left( \sum_{\substack{z<p\leq F \\ p\mid\gcd (A,B,C) }} \frac{\rho_{f}(p^{2})}{p^{2}} + \frac{\omega (\Delta_{f})}{z} + \frac{4}{z} \right) \\
        & \ \ \ \ -\left(\sum_{\substack{z<p\leq F \\ p\mid\gcd (A,B,C)}}\rho_{f}(p^{2}) + \sum_{\substack{z<p\leq F \\ p\nmid\gcd (A,B,C) \\ p\mid \Delta_{f}}} p + 2(\pi (F) - \pi (z))\right),
    \end{align*}
    where $F:=\max_{x<n\leq y} |f(n)|^{1/2}$, and
    \begin{align*}
        S(x,y;f) & \leq (y-x)\prod_{\substack{p\leq z \\ p\mid \Delta_{f}}} \left(1 - \frac{1}{p}\right) \prod_{\substack{p\leq z \\ p\mid\gcd (A,B,C)}} \left(1 - \frac{\rho_{f} (p^{2})}{p^{2}}\right) + \rad(\Delta_{f})\cdot 3^{\pi (z)}.
    \end{align*}
\end{theorem}
\begin{proof}
    For $z>1$ denote
    \begin{align*}
        P(z) := \prod_{p\leq z}p.
    \end{align*}
    
    Let $F:=\max_{x<n\leq y} |f(n)|^{1/2}$. Since $\mu(n)^{2} = \sum_{d^{2}\mid n} \mu(d)$, we have
    \begin{align*}
        S(x,y;f) & = \sum_{x<n\leq y} \mu (f(n))^{2} = \sum_{x<n\leq y}\ \sum_{d^{2}\mid f(n)} \mu(d)  \\
        & = \sum_{x<n\leq y}\ \sum_{d^{2}\mid (f(n),P(z)^{2})} \mu(d) + \sum_{x<n\leq y}\ \sum_{\substack{d^{2}\mid f(n) \\ \exists_{p>z}\ p\mid d}} \mu(d) \\
        & = \sum_{x<n\leq y}\ \sum_{d^{2}\mid (f(n),P(z)^{2})} \mu(d) + \sum_{p>z}\ \sum_{\substack{x<n\leq y \\ p^{2}\mid f(n)}}\ \sum_{\substack{p^{2}d^{2}\mid f(n)}} \mu(pd) \\
        & = \sum_{x<n\leq y}\ \sum_{d^{2}\mid (f(n),P(z)^{2})} \mu(d) - \sum_{z<p\leq F}\ \sum_{\substack{x<n\leq y \\ p^{2}\mid f(n)}}\ \sum_{\substack{d^{2}\mid \frac{f(n)}{p^{2}}}} \mu(d) \\ 
        & = \sum_{x<n\leq y}\ \sum_{d^{2}\mid (f(n),P(z)^{2})} \mu(d) - \sum_{z<p\leq F}\ \sum_{\substack{x<n\leq y \\ p^{2}\mid f(n)}} \left|\mu\left(\frac{f(n)}{p^{2}}\right)\right| \\
        & \geq \sum_{x<n\leq y}\ \sum_{d^{2}\mid (f(n),P(z)^{2})} \mu(d) - \sum_{z<p\leq F} N_{p}(x,y;f).
    \end{align*}

    Now we need to find an upper bound for $N_{p}(z,y;f)$. Denote $x_{p}:=\left\lfloor\frac{x}{p^{2}}\right\rfloor$ and $y_{p}:=\left\lceil\frac{y}{p^{2}}\right\rceil$. Then:
    \begin{align*}
        N_{p}(x,y;f) & = \#\big\{ x<n\leq y | f(n)\equiv 0 \pmod{p^{2}} \big\} \\
        & \leq \#\big\{ p^{2}x_{p}\leq n\leq p^{2}y_{p}-1|f(n)\equiv 0 \pmod{p^{2}} \big\} \\
        & = (y_{p}-x_{p}) \#\big\{ 0\leq n\leq p^{2}-1| f(n)\equiv 0 \pmod{p^{2}} \big\} \\
        & = (y_{p}-x_{p}) \rho_{f}\left(p^{2}\right) \leq \left(\frac{y-x}{p^{2}} + 2\right)\rho_{f}\left(p^{2}\right).
    \end{align*}

    Therefore,
    \begin{align*}
        \sum_{z<p\leq F} N_{p}(x,y;f) & \leq (y-x)\sum_{z<p\leq F} \frac{\rho_{f}(p^{2})}{p^{2}} + 2\sum_{z<p\leq F}\rho_{f}(p^{2}).
    \end{align*}

    We bound the sums above separately. Begin with the  first one:
    \begin{align*}
        \sum_{z<p\leq F} \frac{\rho_{f}(p^{2})}{p^{2}} & \leq \sum_{\substack{z<p\leq F \\ p\mid\gcd (A,B,C) }} \frac{\rho_{f}(p^{2})}{p^{2}} + \sum_{\substack{z<p\leq F \\ p\nmid\gcd (A,B,C) \\ p\mid \Delta_{f}}} \frac{1}{p} + \sum_{\substack{z<p\leq F \\ p\nmid\gcd (A,B,C) \\ p\nmid \Delta_{f}}} \frac{2}{p^{2}} \\
        & \leq \sum_{\substack{z<p\leq F \\ p\mid\gcd (A,B,C) }} \frac{\rho_{f}(p^{2})}{p^{2}} + \frac{\omega (\Delta_{f})}{z} + 2\sum_{\substack{z<p\leq F \\ p\nmid\gcd (A,B,C) \\ p\nmid \Delta_{f}}} \frac{1}{p^{2}}.
    \end{align*}
    The function $1/x^{2}$ is monotonically decreasing, so for $Z>F$ we have
    \begin{align*}
        \sum_{\substack{z<p\leq F \\ p\nmid\gcd (A,B,C) \\ p\nmid \Delta_{f}}} \frac{1}{p^{2}} & < \sum_{z<n \leq Z} \frac{1}{n^{2}} \leq \frac{1}{z^{2}} + \int_{z}^{Z}\frac{\dt}{t^{2}} = \frac{1}{z^{2}} + \frac{1}{z} - \frac{1}{Z} < \frac{2}{z}. 
    \end{align*}
    Hence,
    \begin{align*}
        \sum_{z<p\leq F} \frac{\rho_{f}(p^{2})}{p^{2}} & \leq \sum_{\substack{z<p\leq F \\ p\mid\gcd (A,B,C) }} \frac{\rho_{f}(p^{2})}{p^{2}} + \frac{\omega (\Delta_{f})}{z} + \frac{4}{z}.
    \end{align*}

    For the second sum we have:
    \begin{align*}
        \sum_{z<p\leq F}\rho_{f}(p^{2}) & \leq \sum_{\substack{z<p\leq F \\ p\mid\gcd (A,B,C)}}\rho_{f}(p^{2}) + \sum_{\substack{z<p\leq F \\ p\nmid\gcd (A,B,C) \\ p\mid \Delta_{f}}} p + 2(\pi (F) - \pi (z))
    \end{align*}
    
    Therefore, we obtained the bound for the sum $\sum_{z<p\leq F} N_{p}(x,y;f)$. It remains to deal with the main term:
    \begin{align*}
        \sum_{x<n\leq y}\ \sum_{d^{2}\mid (f(n),P(z)^{2})} \mu(d) & = \sum_{p\mid P(z)} \mu(d) \sum_{\substack{ x<n\leq y \\ f(n)\equiv 0\pmod{d^{2}} }} 1 \\
        & = \sum_{d\mid P(z)} \mu(d) \left(\frac{\rho_{f}(d^{2})}{d^{2}}(y-x) + O_{\leq 1}\left(\rho_{f}(d^{2})\right)\right) \\
        & = (y-x)\sum_{d\mid P(z)}\frac{\mu (d) \rho_{g}(d^{2})}{d^{2}} + O_{\leq 1}\left(\sum_{d\mid P(z)}\rho_{f}(d^{2})\right) \\
        & = (y-x) \prod_{p\leq z} \left(1 - \frac{\rho_{f} (p^{2})}{p^{2}}\right) + O_{\leq 1}\left(\sum_{d\mid P(z)}\rho_{f}(d^{2})\right).
    \end{align*}

    In order to simplify the product, we use previously obtained formulas for $\rho_{f}(p^{2})$:
    \begin{align*}
        \prod_{p\leq z} \left(1 - \frac{\rho_{f} (p^{2})}{p^{2}}\right) & = \prod_{\substack{p\leq z \\ p\mid\gcd (A,B,C)}} \left(1 - \frac{\rho_{f} (p^{2})}{p^{2}}\right) \cdot \prod_{\substack{p\leq z \\ p\nmid \Delta_{f}}} \left(1 - \frac{\rho_{f} (p^{2})}{p^{2}}\right) \cdot \prod_{\substack{p\leq z \\ p\mid \Delta_{f}}} \left(1 - \frac{p}{p^{2}}\right) \\
        & \geq \prod_{\substack{p\leq z \\ p\mid\gcd (A,B,C)}} \left(1 - \frac{\rho_{f} (p^{2})}{p^{2}}\right) \cdot \prod_{\substack{p\leq z \\ p\nmid \Delta_{f}}} \left(1 - \frac{2}{p^{2}}\right) \cdot \prod_{\substack{p\leq z \\ p\mid \Delta_{f}}} \left(1 - \frac{1}{p}\right) \\
        & \geq \prod_{\substack{p\leq z \\ p\mid \Delta_{f}}} \left(1 - \frac{1}{p}\right) \cdot \prod_{\substack{p\leq z \\ p\mid\gcd (A,B,C)}} \left(1 - \frac{\rho_{f} (p^{2})}{p^{2}}\right) \cdot \prod_{p\in\mathbb{P}} \left(1 - \frac{2}{p^{2}}\right) .
    \end{align*}

    The bound for the $O$-term follows easily from Lemma \ref{Lemrho(d2)}:
    \begin{align*}
        \sum_{d\mid P(z)}\rho_{f}(d^{2}) & \leq \rad (\Delta_{f}) \sum_{d\mid P(z)} 2^{\omega (d)} = \rad (\Delta_{f}) \cdot 3^{\pi (z)}.
    \end{align*}
    Taking everything together we obtain the lower bound. Proof of the upper bound is analogous but easier.
\end{proof}

\begin{lemma}\label{LemOmegaUpperBound}
    Let $\omega (n)$ denote the number of distinct prime divisors of $n$.
    \begin{enumerate}
        \item $\omega (n)<2\log n$,
        \item if $n>2$ then $\omega (n) < 3\frac{\log n}{\log\log n}$.
    \end{enumerate}
\end{lemma}
\begin{proof}
    \begin{enumerate}
        \item If $n=p^{v_{1}}\cdots p_{\omega(n)}^{v_{\omega (n)}}$, then
        \begin{align*}
            \log n =\sum_{i=1}^{\omega (n)} v_{i}\log p_{i} \geq \omega (n)\log 2 > \frac{\omega (n)}{2}.
        \end{align*}
        \item Follows from Robin's bound \cite[Th\'{e}or\`{e}me 11]{Robin}.\qedhere
    \end{enumerate}
\end{proof}

\begin{lemma}\label{LemPhi}
    We have for every $N$:
    \begin{align*}
        \sum_{\substack{1\leq n\leq N \\ n\equiv 0~(\mathrm{mod}~3)}} \varphi (n) & = \frac{3}{4\pi^{2}}N^{2} + O(N\log N), \\
        \sum_{\substack{1\leq n\leq N \\ n\equiv 1~(\mathrm{mod}~3)}} \varphi (n) & = \frac{9}{8\pi^{2}}N^{2} + O(N\log N).
    \end{align*}
\end{lemma}
\begin{proof}
    At first, note that
    \begin{align*}
        \sum_{\substack{1\leq d\leq \infty \\ 3\nmid d}} \frac{\mu (d)}{d^{2}} = \prod_{p\neq 3}\left(1-\frac{1}{p^{2}}\right) = \frac{9}{8}\prod_{p}\left(1-\frac{1}{p^{2}}\right) = \frac{9}{8} \cdot\frac{6}{\pi^{2}} = \frac{27}{4\pi^{2}}.
    \end{align*}

    If $\alpha\in\{0,1\}$, we have:
    \begin{align*}
        \sum_{\substack{1\leq n\leq N \\ n\equiv \alpha~(\mathrm{mod}~3)}} \varphi (n) & = \sum_{\substack{1\leq n\leq N \\ n\equiv \alpha~(\mathrm{mod}~3)}} \sum_{d\mid n}\mu(d)\frac{n}{d} = \sum_{d\leq N}\mu(d) \sum_{\substack{q\leq N/d \\ dq\equiv \alpha~(\mathrm{mod}~3)}} q \\
        & = \sum_{\substack{d\leq N \\ 3\mid d}}\mu(d) \sum_{\substack{q\leq N/d \\ 0\equiv \alpha~(\mathrm{mod}~3)}} q + \sum_{\substack{d\leq N \\ 3\nmid d}}\mu(d) \sum_{\substack{q\leq N/d \\ dq\equiv\alpha ~(\mathrm{mod}~3)}} q.
    \end{align*}
    If $\alpha =0$, we then get:
    \begin{align*}
        \sum_{\substack{1\leq n\leq N \\ n\equiv 0~(\mathrm{mod}~3)}} \varphi (n) & = \sum_{\substack{d\leq N \\ 3\mid d}}\mu(d) \sum_{q\leq N/d} q + \sum_{\substack{d\leq N \\ 3\nmid d}}\mu(d) \sum_{\substack{q\leq N/d \\ 3\mid q}} q \\
        & = \sum_{\substack{d\leq N \\ 3\mid d}}\mu(d) \left(\frac{1}{2}\left\lfloor\frac{N}{d}\right\rfloor^{2} + \frac{1}{2}\left\lfloor\frac{N}{d}\right\rfloor\right) + 3\sum_{\substack{d\leq N \\ 3\nmid d}}\mu(d) \sum_{t\leq \frac{N}{3d}} t \\
        & = \frac{N^{2}}{2}\sum_{\substack{d\leq N \\ 3\mid d}}\frac{\mu (d)}{d^{2}} + O\left(N\sum_{\substack{d\leq N \\ 3\mid d}}\frac{1}{d}\right) + 3\sum_{\substack{d\leq N \\ 3\nmid d}}\mu(d) \left(\frac{1}{2}\left\lfloor\frac{N}{3d}\right\rfloor^{2} + \frac{1}{2}\left\lfloor\frac{N}{3d}\right\rfloor\right) \\
        & = \left(\sum_{\substack{d\leq N \\ 3\mid d}}\frac{\mu (d)}{d^{2}} + \frac{1}{3}\sum_{\substack{d\leq N \\ 3\nmid d}}\frac{\mu (d)}{d^{2}}\right)\frac{N^{2}}{2} + O\left(N\log N\right) \\
        & = \left(\sum_{d\leq N}\frac{\mu (d)}{d^{2}} - \frac{2}{3}\sum_{\substack{d\leq N \\ 3\nmid d}}\frac{\mu (d)}{d^{2}}\right)\frac{N^{2}}{2} + O\left(N\log N\right) \\
        & = \left(\sum_{1\leq d\leq \infty}\frac{\mu (d)}{d^{2}} - \frac{2}{3}\sum_{\substack{1\leq d\leq \infty \\ 3\nmid d}}\frac{\mu (d)}{d^{2}}\right)\frac{N^{2}}{2} + O\left(N\log N\right) \\
        & = \left(\frac{3}{\pi^{2}} - \frac{1}{3}\sum_{\substack{1\leq d\leq \infty \\ 3\nmid d}}\frac{\mu(d)}{d^{2}}\right)N^{2} + O(N\log N) \\
        & = \left(\frac{3}{\pi^{2}} - \frac{1}{3}\cdot \frac{27}{4\pi^{2}}\right)N^{2} + O(N\log N) \\
        & = \frac{3}{4\pi^{2}} N^{2} + O(N\log N).
    \end{align*}
    Similarly, if $\alpha =1$, then:
    \begin{align*}
        \sum_{\substack{1\leq n\leq N \\ n\equiv \alpha~(\mathrm{mod}~3)}} \varphi (n) & = \sum_{\substack{d\leq N \\ 3\nmid d}} \mu(d) \sum_{\substack{q\leq N/d \\ q\equiv d~(\mathrm{mod}~3)}} q \\
        & = \sum_{\substack{d\leq N \\  d\equiv 1~(\mathrm{mod}~3)}} \mu(d)\sum_{t\leq \frac{N}{3d}-1} (3t+1) + \sum_{\substack{d\leq N \\  d\equiv 2~(\mathrm{mod}~3)}} \mu(d)\sum_{t\leq \frac{N}{3d}-2} (3t+2) \\
        & = \frac{3}{2}\sum_{\substack{d\leq N \\  d\equiv 1~(\mathrm{mod}~3)}} \mu(d)\left(\left\lfloor \frac{N}{3d}+1 \right\rfloor^{2} + \left\lfloor \frac{N}{3d}+1 \right\rfloor\right) \\
        & \hspace{3cm}+ \frac{3}{2}\sum_{\substack{d\leq N \\  d\equiv 2~(\mathrm{mod}~3)}} \mu(d)\left(\left\lfloor \frac{N}{3d}+2 \right\rfloor^{2} + \left\lfloor \frac{N}{3d}+2 \right\rfloor\right) \\
        & = \frac{N^{2}}{6} \sum_{\substack{d\leq N \\ 3\nmid d}}\frac{\mu (d)}{d^{2}} + O(N\log N) = \frac{N^{2}}{6} \sum_{\substack{d\leq \infty \\ 3\nmid d}}\frac{\mu (d)}{d^{2}} + O(N\log N) \\
        & = \frac{9}{8\pi^{2}}N^{2} + O(N\log N),
    \end{align*}
    and we are done.
\end{proof}

\begin{lemma}\label{LemFTconstant}
    The following inequality holds:
    \begin{align*}
        \prod_{p\in\mathbb{P}}\left(1-\frac{2}{p^{2}}\right) > 0.32.
    \end{align*}
\end{lemma}
\begin{proof}
    The product from the statement is related to the Feller-Tornier constant:
    \begin{align*}
        C_{FT} := \frac{1}{2} + \frac{1}{2}\prod_{p\in\mathbb{P}}\left(1-\frac{2}{p^{2}}\right),
    \end{align*}
    which is known to satisfy $C_{FT}>0.66$, see page 106 of \cite{Finch}. A simple computation yields the result.
\end{proof}

\begin{proof}[Proof of Theorem \ref{ThmCubicSquarefreeDelta}]
    We follow the proof of Theorem \ref{ThmCubic}. Recall that we consider polynomials of the form $f(X) = X^{3}-A_{1}X^{2}+A_{2}X-A_{3}$, and for every $A_{1}$ and $A_{2}$ understand its discriminant as a polynomial in $A_{3}$:
    \begin{align*}
        \Delta_{A_{2}} (A_{3}) := -27A_{3}^{2} + (-4A_{1}^{3} + 18A_{1}A_{2})A_{3} + (A_{1}^{2}A_{2}^{2}-4A_{2}^{3}).
    \end{align*}

    We obtain, from the proof of Theorem \ref{ThmCubic}, that
    \begin{align*}
       \#\mathcal{P}_{3}^{\square +}(A) & = \sum_{A_{2}\leq A^{2}/4} S\big(0,G_{+}(A_{2});\Delta_{A_{2}}\big) + \sum_{A^{2}/4<A_{2}\leq A^{2}/3} S\big(G_{-}(A_{2}),G_{+}(A_{2});\Delta_{A_{2}}\big).
    \end{align*}

    It is easy to check that if some prime $p$ divides the greatest common divisor of the coefficients of $\Delta_{A_{2}}$, then $p=3$. Moreover, it is possible only if $A_{2}$ and $A_{1}=A$ are both divisible by $3$, and then
    \begin{align*}
        \Delta_{A_{2}}(X) & = 27 \left[-X^{2} + \left(-4\left(\frac{A}{3}\right)^{3} + 6\cdot\frac{A}{3}\cdot\frac{A_{2}}{3}\right)X + \left(3\left(\frac{A}{3}\right)^{2}\left(\frac{A_{2}}{3}\right)^{2} - 4\left(\frac{A_{2}}{3}\right)^{3}\right)\right],
    \end{align*}
    so it cannot be square-free. Thus we can assume that there are no primes dividing the greatest common divisor of the coefficients of $\Delta_{A_{2}}$. 

    Recall that the discriminant of $\Delta_{A_{2}}$ is $16 (A^{2}-3A_{2})^{3}$. Hence, Theorem \ref{ThmSquarefreeQuadratic} implies for every $A_{2}$ and all $x,z<y$:
    \begin{align*}
        S(x,y;\Delta_{A_{2}}) & \geq (y-x) \prod_{\substack{p\leq z \\ p\mid 2(A^{2}-3A_{2})}} \left(1 - \frac{1}{p}\right)   \prod_{p\in\mathbb{P}} \left(1 - \frac{2}{p^{2}}\right)  - \rad(2(A^{2}-3A_{2}))\cdot 3^{\pi (z)} \\
        & \ \ \ \ - (y-x)\left( \frac{\omega (2(A^{2}-3A_{2}))}{z} + \frac{4}{z} \right) - \left(\sum_{\substack{z<p\leq F_{A_{2}} \\ p\mid 2(A^{2}-3A_{2})}} p + 2(\pi (F_{A_{2}}) - \pi (z))\right) \\
        & \geq (y-x) \prod_{\substack{p\leq z \\ p\mid 2(A^{2}-3A_{2})}} \left(1 - \frac{1}{p}\right)   \prod_{p\in\mathbb{P}} \left(1 - \frac{2}{p^{2}}\right)  - \rad(2(A^{2}-3A_{2}))\cdot 3^{\pi (z)} \\
        & \ \ \ \ - 8(y-x)\frac{\omega (2(A^{2}-3A_{2}))}{z} - \big(2(A^{2}-3A_{2})\omega(2(A^{2}-3A_{2})) + 2\pi (F_{A_{2}})\big).
    \end{align*}
    Here
    \begin{align*}
        F_{A_{2}} \leq \left(\max \Delta_{A_{2}}\right)^{1/2} = \frac{2}{3\sqrt{3}}\left(A_{1}^{2}-3A_{2}\right)^{3/2}. 
    \end{align*}
    
    Therefore,
    \begin{align*}
        & \sum_{A^{2}/4<A_{2}\leq A^{2}/3} S\big(G_{-}(A_{2}),G_{+}(A_{2});\Delta_{A_{2}}(A_{3})\big) \\
        & \geq \frac{4}{27} \prod_{p\in\mathbb{P}} \left(1 - \frac{2}{p^{2}}\right) \sum_{A^{2}/4<A_{2}\leq A^{2}/3} \left(A^{2}-3A_{2}\right)^{3/2} \prod_{\substack{p\leq z \\ p\mid 2(A^{2}-3A_{2})}} \left(1 - \frac{1}{p}\right) \\
        & \ \ \ \ - 3^{\pi (z)} \sum_{A^{2}/4<A_{2}\leq A^{2}/3} \rad(2(A^{2}-3A_{2})) - \frac{32}{27}\cdot \frac{1}{z}\sum_{A^{2}/4<A_{2}\leq A^{2}/3} \left(A^{2}-3A_{2}\right)^{3/2}\omega(2(A^{2}-3A_{2})) \\
        & \ \ \ \ - 2\sum_{A^{2}/4<A_{2}\leq A^{2}/3}(A^{2}-3A_{2})\omega(2(A^{2}-3A_{2})) -2 \sum_{A^{2}/4<A_{2}\leq A^{2}/3} \pi\left(\frac{4}{3\sqrt{3}}\left(A_{1}^{2}-3A_{2}\right)^{3/2}\right).
    \end{align*}
    Let us bound the above sums starting with the last one. \cite[Corollary 1] {RS62} gives $\pi(x)<2\frac{x}{\log x}$ for all $x$, so
    \begin{align*}
        \sum_{A^{2}/4<A_{2}\leq A^{2}/3} \pi\left(\frac{4}{3\sqrt{3}}\left(A_{1}^{2}-3A_{2}\right)^{3/2}\right) & \leq \frac{A^{2}}{3}\pi\left(A^{3}\right)<\frac{2}{9} \frac{A^{5}}{\log A}.
    \end{align*}
    For the second one we will use Lemma \ref{LemOmegaUpperBound}(1):
    \begin{align*}
        \sum_{A^{2}/4<A_{2}\leq A^{2}/3}(A^{2}-3A_{2})\omega(2(A^{2}-3A_{2}))  & \leq \frac{A^{2}}{3} \cdot A^{2}\omega\left(A^{3}\right) < 2A^{4}\log A. 
    \end{align*}
    Similarly for the next one:
    \begin{align*}
        \sum_{A^{2}/4<A_{2}\leq A^{2}/3} \left(A^{2}-3A_{2}\right)^{3/2}\omega(2(A^{2}-3A_{2})) & \leq \frac{A^{2}}{3}\cdot A^{3}\cdot  \omega\left(A^{3}\right) < A^{5} \frac{\log A}{\log\log A}.
    \end{align*}
    Finally, we have
    \begin{align*}
        \sum_{A^{2}/4<A_{2}\leq A^{2}/3} \rad(2(A^{2}-3A_{2})) & \leq \frac{A^{2}}{3}\cdot  2A^{2} = \frac{2}{3}A^{4}.
    \end{align*}
    Therefore, the whole error term can be bounded by
    \begin{align*}
        \frac{2}{3}\cdot 3^{\pi(z)}A^{4} +  \frac{16}{27} \cdot\frac{1}{z}\cdot A^{5}\frac{\log A}{\log\log A} + 4A^{4}\log A + \frac{4}{9} \frac{A^{5}}{\log A}. 
    \end{align*}
    By the bound $\pi(z)\leq 2\frac{z}{\log z}$ and by choosing for example $z=\frac{1}{3}\log A\log\log A$ we see that the above expression is bounded from above by:
    \begin{align*}
        \frac{2}{3}A^{4 + \frac{\log 9}{3}} + 16 \frac{A^{5}}{\left(\log\log A\right)^{2}} + 4A^{4}\log A + \frac{4}{9}\frac{A^{5}}{\log A}\ll \frac{A^{5}}{(\log\log A)^{2}}.
    \end{align*}

    It remains to find a lower bound for the main term. Lemma \ref{LemPhi} implies:
    \begin{align*}
        \sum_{A^{2}/4<A_{2}\leq A^{2}/3} & \left(A^{2} - 3A_{2}\right)^{3/2} \prod_{\substack{p\leq z \\ p\mid 2(A^{2}-3A_{2})}} \left(1 - \frac{1}{p}\right)= \sum_{\substack{1\leq n< \frac{1}{4}A^{2} \\ n\equiv A^{2}~(\mathrm{mod}~3)}} n^{3/2} \prod_{\substack{p\leq z \\ p\mid 2n}} \left(1 - \frac{1}{p}\right) \\
        & \geq \frac{1}{2} \sum_{\substack{\frac{1}{9}A^{2}\leq n< \frac{1}{4}A^{2} \\ n\equiv A^{2}~(\mathrm{mod}~3)}} n^{3/2} \prod_{\substack{ p\mid n}} \left(1 - \frac{1}{p}\right) \geq \frac{A}{6} \sum_{\substack{\frac{1}{9}A^{2}\leq n< \frac{1}{4}A^{2} \\ n\equiv A^{2}~(\mathrm{mod}~3)}} \varphi (n) \\
        & \geq \frac{1}{6}\cdot \frac{3}{4\pi^{2}}\left(\frac{1}{16}-\frac{1}{81}\right)A^{5} + O\left(A^{4}\log A\right).
    \end{align*}
    Therefore, the constant of the main term $A^{5}$ is
    \begin{align*}
        \frac{4}{27}\cdot \frac{1}{6}\cdot \frac{3}{4\pi^{2}}\left(\frac{1}{16}-\frac{1}{81}\right)\cdot \prod_{p\in\mathbb{P}}\left(1-\frac{2}{p^{2}}\right) >3\cdot 10^{-5}
    \end{align*}
    by Lemma \ref{LemFTconstant}. We are done.
\end{proof}

\section{Proof of Theorem \ref{ThmGeneralRoot}}

\begin{lemma}\label{LemTao}
    Let $(s_{1},\ldots ,s_{n})$ be a sequence of integers such that the polynomial 
    \begin{align*}
        X^n+\sum^n_{k=1}(-1)^k\binom{n}{k}s_kX^{n-k}
    \end{align*}
    has all its roots real. Then we have 
    \begin{align*}
        |s_{k}|^{1/k}\leq C \max_{j\in\{l,l+1\}}\left(\frac{k}{j}\right)^{1/2} |s_{j}|^{1/j}, 
    \end{align*}
    for all $1\leq l<k\leq n$, where $C=160e^{7}$.
\end{lemma}
\begin{proof}
    This is Theorem 1.1 in \cite{Tao}.
\end{proof}

\begin{proof}[Proof of Theorem \ref{ThmGeneralRoot}]
    Apply Lemma \ref{LemTao} with $l=1$ and $s_{k}:=\frac{1}{\binom{n}{k}}A_{k}$. We immediately obtain that for every $k\geq 3$:
    \begin{align*}
        |A_{k}|\leq C^{k} k^{k/2} \max\left\{\frac{1}{\binom{n}{1}}|A_{1}|,\left(\frac{1}{2\binom{n}{2}}|A_{2}|\right)^{1/2}\right\}^{k}= C^{k} k^{k/2} M(A_{1},A_{2})^{k}.
    \end{align*}
    Therefore, for every $k$ we have at most
    \begin{align}
        2C^{k} k^{k/2} M(A_{1},A_{2})^{k}+1\leq 3C^{k} k^{k/2} M(A_{1},A_{2})^{k}
    \end{align}
    choices of $A_{k}$. By multiplying the above numbers, we get the result.
\end{proof}

\section{Proof of Theorem \ref{ThmUpBoundW1}}

\begin{proof}[Proof of Theorem \ref{ThmUpBoundW1}]
    Let $f(X)=X^{3}-AX^{2}+BX-A_{3}$. From the assumptions we have
    \begin{align*}
        \Delta_{f}-D = -27A_{3}^{2} + (-4A^{3}+18AB)A_{3}+(A^{2}B^{2}-4B^{3}) -D \leq 0.
    \end{align*}
    Let us consider the above expression as a polynomial in $A_{3}$. This gives two possibilities depending on the value of
    \begin{align*}
        \Delta := \Delta_{\Delta_{f}-D} = 16(A^2-3B)^{3} - 4\cdot 27D.
    \end{align*}
    \begin{itemize}
        \item if $\Delta\leq 0$, then $\Delta_{f}\leq D$ for all possible values of $A_{3}$. By \eqref{IneqCubicA3} we have
        \begin{align*}
            \mathcal{W}(A,B;D)\leq \lfloor G_{+}(A,B)\rfloor - \lceil G_{-}(A,B)\rceil \leq \frac{4}{27}\left(A^{2}-3B\right)^{3/2}.
        \end{align*}
        \item if $\Delta >0$ then either
        \begin{align*}
            A_{3}\leq \frac{1}{27}\left(9AB-2A^{3}-2\sqrt{\left(A_{1}-3B\right)^{3}-\frac{27}{4}D}\right)
        \end{align*}
        or
        \begin{align*}
            A_{3}\geq \frac{1}{27}\left(9AB-2A^{3}+2\sqrt{\left(A_{1}-3B\right)^{3}-\frac{27}{4}D}\right).
        \end{align*}
        Therefore, thanks to \eqref{IneqCubicA3}, we obtain
        \begin{align*}
            \mathcal{W}(A,B;D)\leq \frac{4}{27}\left(\left(A^{2}-3B\right)^{3/2} - \left(\left(A^{2}-3B\right)^{3}-\frac{27}{4}D\right)^{1/2}\right) < \frac{D}{\left(A^{2}-3B\right)^{3/2}}.
        \end{align*}
        The last inequality follows from the following trivial computations which are valid for all positive numbers $Z,T$ with $T<Z$:
        \begin{align*}
            Z^{1/2}-(Z-T)^{1/2} = \frac{T}{Z^{1/2}+(Z-T)^{1/2}} < \frac{T}{Z^{1/2}}.
        \end{align*}
    \end{itemize}
    It is easy to check that in both cases the expression obtained is less than or equal to $\frac{2}{3\sqrt{3}}\sqrt{D}$. The result follows.
\end{proof}
\begin{remark}
     Theorem \ref{ThmUpBoundW1} can be used to estimate the number of cubic monic integral polynomials with only real roots and bounded height and discriminant. Indeed, the number of polynomials $f(X)=X^{3}-AX^{2}+BX-C$ with only real roots and such that $|A|,|B|,|C|\leq H$ and $\Delta_{f}\leq D$ is bounded from above by
     \begin{align*}
         \sum_{|A|\leq H}\ \sum_{|B|\leq H} \#\mathcal{P}_3(A,B;D)\leq \frac{2}{3\sqrt{3}}D^{1/2} (2H+1)^{2} \leq 2\sqrt{3} H^{2} D^{1/2}
     \end{align*}
     if $H\geq 1$. This can be compared with the bound $\leq c(\varepsilon)H^{2/3+\varepsilon}D^{5/6}$ in \cite[Theorem 13]{Badziahin} for the number of all cubic polynomials with coefficients $\leq H$ and discriminant $\leq D$, where $c(\varepsilon)>0$ is a positive constant depending on $\varepsilon$. The bound from Theorem \ref{ThmUpBoundW1} is better if $D\geq \frac{24\sqrt{3}}{c(\varepsilon)^{3}} H^{4-3\varepsilon}$. The latter is true if $D\geq H^{4}$ and $H$ is large enough. 
\end{remark}

\section{Cubic polynomials with almost prime discriminants}

\begin{theorem}\label{ThmAlmostPrimeDisc}
    Let $A$ and $B$ be integers such that $A^{2}-3B$ is not a square. Then there exist
    \begin{align*}
        \gg \frac{H}{(\log H)^{2}}
    \end{align*}
    values of $C\in \{-H,-H+1,\ldots,H-1,H\}$ such that the discriminant of the polynomial $f(X):=X^{3}-AX^{2}+BX-C$ has at most two prime divisors. The implied constant is universal.
\end{theorem}
\begin{proof}
    Recall that the discriminant of $f(X)$ is the following quadratic polynomial in $C$:
    \begin{align*}
        \Delta_{f}(C) = -27C^{2} + (-4A^{3}+18AB)C + (A^{2}B^{2}-4B^{3}).
    \end{align*}
    Its discriminant is $\Delta :=\Delta_{\Delta_{f}} = 16(A^{2}-3B)^{3}$. Therefore, the assumptions on $A$ and $B$ imply that the polynomial $\Delta_{f}$ is irreducible. Therefore, we can apply the main results in the papers \cite{Iwaniec,Oliver}. We obtain the following estimate:
    \begin{align*}
        \gg \prod_{p\in\mathbb{P}}\frac{1-\frac{\rho (p)}{p}}{1-\frac{1}{p}} \cdot \frac{H}{\log H},
    \end{align*}
    where $\rho_f(p)$ denotes the number of incongruent solutions of congruence $\Delta_{f}(X)\equiv 0\pmod{p}$. Now note that $\rho_f(2)=1$ since $\Delta_{f}(C)\equiv (C-AB)^{2} \pmod{2}$, so the product appearing in the estimate is nonzero. Moreover, for every prime $p$, $\rho_f(p)\leq 2$. Hence,
    \begin{align*}
        \prod_{p\in\mathbb{P}}\frac{1-\frac{\rho (p)}{p}}{1-\frac{1}{p}} \cdot \frac{H}{\log H}\geq \prod_{\substack{p\in\mathbb{P} \\ p\geq 3}}\frac{1-\frac{2}{p}}{1-\frac{1}{p}} \cdot \frac{H}{\log H} \gg \frac{H}{(\log H)^{2}},
    \end{align*}
    as desired.
\end{proof}

The question arises of how restrictive the condition that $A^{2}-3B$ is not a square is. It appears that it is not very restrictive and, in fact, it is satisfied for almost all pairs $(A,B)$, which is an immediate corollary of the following proposition:

\begin{prop}\label{PropAlmostPrime}
    For every $H\geq 2$ we have
    \begin{align*}
        \# \left\{ (A,B)\in\mathbb{Z}^{2}\ |\ -H\leq A,B\leq H,\ A^{2}-3B \textrm{ is a square}  \right\} \asymp H\log H.
    \end{align*}
\end{prop}
\begin{proof}
Observe, that
\begin{align*}
    \sum_{\substack{-H\leq A,B\leq H \\ A^{2}-3B \textrm{ is a square}}} 1 & = \sum_{Z^{2}\leq H^{2} + 3H}\ \sum_{\substack{-H\leq A,B\leq H \\ A^{2}-3B=Z^{2}}} 1 = \sum_{Z^{2}\leq H^{2} + 3H}\ \sum_{\substack{-H\leq A\leq H \\ 3\mid A^{2}-Z^{2} \\ |A^{2}-Z^{2}|\leq 3H}} 1 \\
    & = \sum_{-H\leq A\leq H} \sum_{{\substack{Z \\ 3\mid A^{2}-Z^{2} \\ |A^{2}-Z^{2}|\leq 3H}}} 1 = \sum_{-H\leq A\leq H} \ \sum_{\substack{\sqrt{\max\{A^{2}-3H,0\}}\leq Z\leq \sqrt{A^{2}+3H} \\ A^{2}\equiv Z^{2}~(\mathrm{mod}~3)}} 1.
\end{align*}
For the upper bound, we have that the above expression is further:
\begin{align*}
    & \leq \sum_{-H\leq A\leq H} \left(\sqrt{A^{2}+3H} - \sqrt{\max\{A^{2}-3H,0\}}\right) \\
    & = \sum_{\substack{-H\leq A\leq H \\ |A|\geq \sqrt{3H}}} \frac{6H}{\sqrt{A^{2}+3H} + \sqrt{A^{2}-3H}} + \sum_{\substack{-H\leq A\leq H \\ |A| < \sqrt{3H}}} \sqrt{A^{2}+3H} \\
    &< 12H \sum_{\sqrt{3H}\leq A\leq H} \frac{1}{A} + \sqrt{6H}\sum_{-\sqrt{3H}\leq A\leq \sqrt{3H}} 1 \ll H\log H + H \ll H\log H.
\end{align*}
Similarly for the lower bound:
\begin{align*}
    & \geq \sum_{-H/3 \leq A_{1}\leq H/3} \ \sum_{\frac{1}{3}\sqrt{\max\{9A_{1}^{2}-3H,0\}} \leq Z_{1}\leq \frac{1}{3}\sqrt{9A_{1}+3H}} 1 \\
    & \geq \sum_{\substack{-H/3\leq A_{1}\leq H/3 \\ A_{1}^{2}\geq H/3} } \left(\sqrt{9A_{1}^{2}+3H} - \sqrt{9A_{1}^{2}-3H} -2\right) \\
    & = 2\sum_{\sqrt{H/3}\leq A_{1}\leq H/3} \frac{6H}{\sqrt{9A_{1}^{2}+3H} + \sqrt{9A_{1}^{2}-3H}} + O(H) \\
    & \geq 2H \sum_{\sqrt{H/3}\leq A_{1}\leq H/3} \frac{1}{A_{1}} + O(H) \gg H\log H,
\end{align*}
because of the elementary inequality $\sqrt{x+y} + \sqrt{x-y} \leq 2\sqrt{x}$, which is true for all $x\geq y\geq 0$. The result follows.
\end{proof}

Theorem \ref{ThmAlmostPrimeDiscLowBound} follows immediately.

\begin{proof}[Proof of Theorem \ref{ThmAlmostPrimeDiscLowBound}.]
    Let $H>1$. For all integers $A$, $B$, $C$ denote by $\Delta_{A,B,C}$ the discriminant of the polynomial $X^{3}-AX^{2}+BX-C$.
    
    Theorem \ref{ThmAlmostPrimeDisc} and Proposition \ref{PropAlmostPrime} imply:
    \begin{align*}
        \#\{(A,B,C)\in  &  \mathbb{Z}^{3} | A,B,C\in[-H,H],\ \omega\left(\Delta\right)\leq 2 \} \geq \sum_{\substack{-H\leq A,B\leq H \\ A^{2}-3B \textrm{ is not a square}}}\ \sum_{\substack{-H\leq C\leq H \\ \omega (\Delta_{A,B,C})\leq 2}} 1 \\
        & \gg \frac{H}{(\log H)^{2}} \#\left\{(A,B)\in \mathbb{Z}^{2} | -H\leq A,B\leq H,\ A^{2}-3B\ \textrm{is not a square} \right\} \\
        & \gg \frac{H^{3}}{(\log H)^{2}},
    \end{align*}
    and we are done.
\end{proof}

\bibliographystyle{alpha}
\bibliography{pub}

\end{document}